\documentclass[12pt, reqno]{amsart}

\usepackage{amssymb, amsmath, amsthm,stmaryrd}
\usepackage[backref]{hyperref}
\usepackage[alphabetic,backrefs,lite]{amsrefs}
\usepackage{amscd}   
\usepackage{fullpage}
\usepackage[all]{xy} 
\usepackage{MnSymbol,graphicx}
\usepackage{mathrsfs,enumitem}

\usepackage[usenames,dvipsnames]{color}

\usepackage{tikz}

\newtheorem{lemma}{Lemma}[section]
\newtheorem{theorem}[lemma]{Theorem}
\newtheorem{proposition}[lemma]{Proposition}

\newtheorem{cor}[lemma]{Corollary}

\newtheorem{claim*}{Claim}

\newtheorem{defn}[lemma]{Definition}

\theoremstyle{definition}
\newtheorem{remark}[lemma]{Remark}

\newtheorem{example}[lemma]{Example}

\newtheorem{assumption}[lemma]{Assumption}
\newtheorem{problem}[lemma]{Problem}

\newcommand{\G}{{\mathbb G}}

\newcommand{\PP}{{\mathbb P}}
\newcommand{\C}{{\mathbb C}}
\newcommand{\F}{{\mathbb F}}

\newcommand{\Q}{{\mathbb Q}}
\newcommand{\R}{{\mathbb R}}
\newcommand{\Z}{{\mathbb Z}}

\newcommand{\calO}{{\mathcal O}}
\newcommand{\cO}{{\mathcal O}}

\newcommand{\frp}{{\mathfrak p}}

\newcommand{\oJ}{\overline{J}}
\newcommand{\oC}{\overline{C}}

\DeclareMathOperator{\Tr}{\bf Tr}

\DeclareMathOperator{\End}{End}

\DeclareMathOperator{\Hom}{Hom}

\DeclareMathOperator{\Aut}{Aut}
\DeclareMathOperator{\Gal}{Gal}

\DeclareMathOperator{\Pic}{Pic}
\DeclareMathOperator{\Jac}{Jac}

\DeclareMathOperator{\Spec}{Spec}

\DeclareMathOperator{\Nrd}{Nrd}
\DeclareMathOperator{\N}{\bf Nm}

\numberwithin{equation}{section}
\numberwithin{table}{section}

\title{Bad Reduction of Genus Three Curves with Complex Multiplication}

\author{Irene Bouw, Jenny Cooley, Kristin Lauter, Elisa Lorenzo Garcia, Michelle Manes, Rachel Newton, Ekin Ozman}

\thanks{The work of MM was partially supported by NSF-DMS 1102858.}

 \makeatletter
\let\@wraptoccontribs\wraptoccontribs
\makeatother


\begin{document}


	\begin{abstract}
	Let $C$ be a smooth, absolutely irreducible genus $3$ curve
        over a number field $M$. Suppose that the Jacobian of $C$ has
        complex multiplication by a sextic CM-field $K$. Suppose
        further that $K$ contains no imaginary quadratic subfield. We
        give a bound on the  primes $\mathfrak{p}$  of $M$ such that the stable
        reduction of $C$ at  $\mathfrak{p}$ contains three irreducible
        components of genus $1$.
	
	\end{abstract}
	\maketitle

\section{Introduction} \label{sec:intro}

In \cite{GorenLauter07}, Goren and Lauter study genus $2$ curves whose
Jacobians
are absolutely simple and
have complex multiplication (CM) by the ring of integers $\mathcal O_K$
 of a quartic CM-field $K$,
and they show that if such a curve has bad reduction to
characteristic $p$ then there is a solution to the embedding problem, formulated as
follows \cite{GorenLauter07}:

Let $K$ be a quartic CM-field which does
 not contain a proper CM-subfield, and let $p$ be a prime.  The embedding problem concerns finding
 a ring embedding $\iota: \mathcal O_K \hookrightarrow \End(E_1 \times
 E_2)$, such that the Rosati involution coming from the product
 polarization induces complex conjugation on $\mathcal O_K$,
and $E_1, E_2$ are supersingular elliptic curves over $\overline{\mathbb F}_p$.

In this paper, we consider genus $3$ curves whose Jacobians have CM
by  a sextic CM-field that does not contain a proper
CM-subfield. By analogy with
\cite{GorenLauter07}, we formulate an
 embedding problem for the genus $3$ case as follows.

\bigskip\noindent {\bf Problem \ref{problem:embedding}}\, (The
embedding problem) Let ${\mathcal O}$ be an order in a sextic CM-field
$K$, and let $p$ be a prime number.  The {\em embedding problem} for
$\cO$ and $p$ is the problem of finding elliptic curves $E_1, E_2, E_3$ defined over
$\overline{\mathbb F}_p$, and a ring embedding
\[
i: {\mathcal O}\hookrightarrow \End(E_1\times E_2\times E_3)
\]
such that the Rosati involution on $\End(E_1\times E_2\times E_3)$ induces
 complex conjugation on ${\mathcal O}$. We call such a ring
embedding a {\em solution to the embedding problem} for $\cO$
and $p$.

\bigskip
In this paper, we prove the following result on solutions to the
embedding problem.  We refer to Section \ref{62} for the precise
statement.

\bigskip\noindent {\bf Theorem \ref{thm:no embedding}}\, {\em Let $K$
  be a sextic CM-field such that $K$ does not contain a proper
  CM-subfield. Let ${\mathcal O}$ be an order in $K$.  There exists an
  explicit bound on the rational primes $p$
  for which the embedding problem has a solution, and
	this bound depends only on the order ${\cO}$.}

\bigskip
As in the genus $2$ case, Theorem \ref{thm:no embedding} yields a
bound on certain primes of bad reduction for the curve $C$. However,
the result is not as strong as in the genus $2$ case, since there are
more possibilities for the reduction of $C$. We  discuss the
statement of the result.

Let $C$ be a smooth, absolutely irreducible genus $3$ curve over a
number field $M$ whose Jacobian has CM by an order $\cO$ in a sextic
CM-field $K$.  We say that $C$ has bad reduction at a rational prime
$p$ if there exists a prime $\mathfrak{p}$ of $M$ above $p$ at which
$C$ has bad reduction.
In Corollary \ref{cor:SerreTate}, we observe that if
$C$ has bad reduction at a prime $\mathfrak{p}$, there are two
possibilities for the stable reduction $\oC_{\mathfrak{p}}$ of $C$ at
$\mathfrak{p}$. Either $\oC_{\mathfrak{p}}$ contains three irreducible
components of genus $1$, or $\oC_{\mathfrak{p}}$ contains one irreducible
component of genus $1$ and one of genus $2$.

In this paper, we restrict our attention to the first of these two
possibilities.  In Proposition \ref{prop:assumptions}, we show that if
$C$ has bad reduction at a prime $\mathfrak{p}$ above $p$ and the
stable reduction contains three genus~1 curves, then the embedding
problem for $\cO$ and $p$ has a solution. Theorem \ref{thm:no
  embedding} therefore yields the following result on the primes of
bad reduction of $C$.

\bigskip\noindent {\bf Theorem \ref{thm:bnd_on_p}}\, {\em Let $C$ be a
  genus $3$ curve whose Jacobian has CM by an order $\cO$ in a sextic
  CM-field $K$ that does not contain a proper CM-subfield. There
  exists an explicit bound on the  primes $p$ where the stable reduction
   contains three irreducible
  components of genus~$1$.}

\bigskip
We do not consider all primes of bad reduction of $C$ in
Theorem \ref{thm:bnd_on_p} for the following reason. If the stable reduction
   of $C$ at $\mathfrak{p}$ contains three irreducible
  components of genus $1$,  then the reduction $\oJ_{\mathfrak{p}}$ of the
Jacobian $J$ of $C$ is isomorphic to the product $E_1\times E_2\times
E_3$ of elliptic curves as polarized abelian varieties (Proposition
\ref{prop:SerreTate}). This yields a ring embedding
\[
\iota:
\cO=\End(J)\hookrightarrow \End(\oJ_{\mathfrak{p}})=\End(E_1\times
E_2\times E_3),
\]
which has the property that the Rosati involution on $\End(E_1\times
E_2\times E_3)$ restricts to complex conjugation on the image of $\cO$
(Section \ref{sec:polarization}). This is precisely the statement that
$\iota$ is a solution to the embedding problem for $\cO$ and $p$.

Consider a prime $\mathfrak{p}$ where the curve $C$ has bad reduction,
but the stable reduction $\oC_\mathfrak{p}$ contains an irreducible
component $E$ of genus $1$ and an irreducible component $D$ of genus
$2$ (Corollary \ref{cor:SerreTate}). In this case --- an example of which is
described in Section \ref{sec:Picardexa} --- the reduction
$\oJ_\mathfrak{p}$ of the Jacobian of $C$ is the product of $E$ with
the Jacobian of $D$ as polarized abelian varieties. The abelian
variety $\oJ_\mathfrak{p}$ is still isogenous to a product of elliptic
curves (Theorem \ref{thm: Lang2}), but $\oJ_\mathfrak{p}$ is not
isomorphic to a product of elliptic curves as polarized abelian
varieties.  This suggests that a different formulation of the
embedding problem would be needed to draw conclusions for such primes
$\mathfrak{p}$.  We do not discuss the correct formulation of the
  embedding problem for this case in the present paper, but leave it
  as a direction for future work.

The assumption that the CM-field $K$ does not contain a proper
CM-field is also present in the genus $2$ case
in \cite{GorenLauter07}. However, in the genus $2$ case, this assumption is
equivalent to the assumption that the CM-type of the Jacobian $J$
is primitive.  We refer to Section \ref{sec:g=2} for more details.  In
characteristic zero, the condition that the CM-type corresponding to
$J$ is primitive is equivalent to the assumption that $J$ is
absolutely simple (Theorem \ref{thm: primsimple}).

In the genus $3$ case, the assumption that the CM-field $K$ does not
contain a proper CM-subfield still implies that the CM-type of the
Jacobian $J$ is primitive. However, the converse does not hold. Even
in the case that the sextic CM-field $K$ contains a proper CM-subfield
there exist primitive CM-types (Section \ref{sec:prim}). In Section
\ref{subsec:degenerate}, we discuss why the embedding problem needs to
be formulated differently for such CM-fields. We show that, in
the case where $K$ contains a proper CM-subfield, the embedding problem
as we have formulated it has solutions for any prime $p$ and some
order ${\mathcal O}$ of $K$. 

Finally, we have not included the
condition that the elliptic curves $E_i$ are supersingular in the
formulation of the embedding problem, in contrast to the formulation
in genus $2$,  because for a set of Dirichlet density $1/2$, the
elliptic curves $E_i$ are ordinary.

\subsection{Relation to a result of Gross and Zagier}

One of the motivations of Goren and Lauter for studying solutions of the
embedding problem in genus $2$ was generalizing a result of
Gross and Zagier on singular moduli of elliptic curves
\cite{GrossZagier}.  Recall that  singular moduli are values
$j(\tau)$ of the modular function $j$ at imaginary quadratic numbers
$\tau$.
Gross and Zagier define the product
\[
J(d_1, d_2)=\left(\prod_{[\tau_1], [\tau_2]}\left(j(\tau_1)-j(\tau_2)\right)\right)^{4/w_1w_2},
\]
where the product runs over equivalence classes of imaginary quadratic
numbers $\tau_i$ with discriminants $d_i$, where the $d_i$ are
assumed to be relatively prime. Here $w_i$ denotes the number of units
in $\Q(\tau_i)$. The function $J$ is closely related to the
value of the Hilbert class polynomial of an imaginary quadratic field
at a point $\tau$ corresponding to a different imaginary quadratic field.

 Under some assumptions, Gross and Zagier show that
$J(d_1, d_2)$ is an integer, and their main result gives a formula for the factorization of
this integer. The result of Gross and Zagier may be reinterpreted as a
formula for the number of isomorphisms between the reductions of the
elliptic curves $E_i$ corresponding to the $\tau_i$ at all rational
primes $p$. This problem is equivalent to counting embeddings of
$\End(E_2)$ into the endomorphism ring of the reduction of $E_1$ at $p$.

Goren and Lauter (\cite{GorenLauter07}, Corollary 5.1.3) prove a
generalization of the result of Gross and Zagier.  They consider curves
of genus $2$ with CM by a quartic CM-field. In their result, the
function $J$ is replaced by suitable Siegel modular functions
$f/\Theta^k$.  Here $f$ is a Siegel modular form of weight $10k$ with
values in a number field and $\Theta$ is a concrete Siegel modular
form of weight $10$. The modular function $f/\Theta^k$ has the
property that for any $\tau$ in the Siegel upper half plane the
genus $2$ curve corresponding to $\tau$ has bad reduction at the
primes dividing the denominator of $(f/\Theta^k)(\tau)$. (See
\cite{GorenLauter07}, Corollary 5.1.2 for the precise statement.)

The Igusa class polynomials are an analog of the Hilbert class
polynomials for quartic CM-fields, where the $j$-invariant is replaced
by the absolute Igusa invariants. Goren and Lauter and collaborators (see
for example \cite{GorenLauter07}, \cite{GorenLauter13},
\cite{LauterViray}) deduce results on the denominators of the
coefficients of the Igusa class polynomials from results on the
embedding problem for quartic CM-fields.

The embedding problem for curves of genus $3$ studied in this
paper does not immediately yield a statement analogous to that of
Gross and Zagier. One of the ingredients that is missing  is
finding good coordinates for the moduli space of curves of genus $3$,
analogous to the absolute Igusa invariants in genus $2$. 

In this paper, we discuss several differences between the reduction of
CM-curves in genus $2$ and in genus $3$. The embedding problem in the
formulation of Problem \ref{problem:embedding} does not cover all
types of bad reduction. Also, in the case that the sextic CM-field $K$
contains a proper CM-subfield the embedding problem should be
adapted. It would be interesting to study the implication of these
differences for a possible analog of the Igusa class polynomials for
sextic CM-fields.

\subsection{Outline}
The structure of this paper is as follows. Section \ref{sec:Galois}
gives the possibilities for the Galois group of the Galois closure of
a sextic CM-field, following work of Dodson in \cite{Dodson}.  Section
\ref{sec:prim} describes the possible CM-types for a sextic
CM-field. We note which of the CM-types are primitive, meaning that
they can arise as the CM-type of a simple abelian variety. In Section
\ref{sec:jacobian3}, we describe the possibilities for the reduction
of a genus $3$ curve and its Jacobian to characteristic $p>0$. We also
give some properties of the Rosati involution attached to a polarized
abelian variety, which will be used in Section \ref{sec:embedding}. In
Section \ref{sec:exa}, we give various examples of genus $3$ curves
with CM; we calculate their CM-types and the reductions of the curves
and their Jacobians to characteristic $p>0$. In Section
\ref{sec:embedding}, we consider a genus $3$ curve $C$ over a number
field $M$ such that its Jacobian has CM by a sextic CM-field $K$ with
no proper CM-subfield. We prove a bound on primes such that there
exists a solution to the embedding problem, and we use that to give a
bound on the primes $p$  such that the stable reduction of
$C$ at $p$ contains three elliptic curves. We show that if
we drop the assumption that $K$ has no proper CM-subfield, then the
embedding problem as stated cannot be used to give a bound on the
primes $p$ as above.

We include as an appendix a collection of conditions that a solution
to the embedding problem must satisfy, written as equations in the
entries of certain matrices in the image of the embedding. These
equations may be useful for future work. A refinement of the embedding
problem (for example, a version which includes conditions pertaining
to the CM-type) would result in extra equations in addition to those
in the appendix. It is to be hoped that studying this larger set of
equations would yield an explicit bound on the primes for which they
have a solution. This would give a bound on the primes $p$ such that
the stable reduction of $C$ at $p$ contains three curves of genus $1$,
even in the case where the CM-field $K$ contains a proper CM-subfield.

\subsection{Notation and conventions}
We set the following notation, to be used throughout.
\begin{itemize}
\item $\mathbb F_p$ is the finite field with $p$ elements.
\item $\zeta_N$ is a primitive $N$th root of unity.
\item For a field $k$, $\overline k$ is an algebraic closure.
\item $K$ is a sextic CM-field, i.e., $K$ is a totally imaginary extension of $K^+$, where $K^+$ is a totally real cubic extension of $\Q$.
\item
 $\mathcal O$ is an order of $K$.
\item $F$ and $L$ are Galois closures of $K/\Q$ and $K^+/\Q$ respectively, with  $G=\Gal(F/\Q)$ and $G^+=\Gal(L/\Q)$.
\item $\psi$ is a complex embedding $K \hookrightarrow
\C$, and $\rho$ is complex conjugation.  Hence $\{\psi, \rho\circ \psi\}$ is a conjugate pair of embeddings.
\item $(K,\varphi)$ is a CM-type, i.e., a choice of one embedding from each pair of complex conjugate embeddings.
\item $A$ is an abelian variety, $\End(A)$ is the endomorphism ring of $A$, and $\End^0(A)$ is $\End(A) \otimes \mathbb Q$.
\item For $f \in \End(A)$, $f^\vee\in \End(A^\vee)$ is the dual
  isogeny. The Rosati involution associated with a fixed polarization
  is denoted by $f\mapsto f^\ast, \, \End^0(A)\to \End^0(A)$.
\item $E$ is an elliptic curve, $j(E)$ is the $j$-invariant of $E$.
\item We denote an isomorphism between two abelian varieties over an
  algebraic closure of the field of definition by $\simeq$.
\item We denote an isogeny between two abelian varieties over an
  algebraic closure of the field of definition by $\sim$.
\item $M$ is a number field, $\nu$ (or $\mathfrak{p}$) is a finite place of $M$,  ${\mathcal O}_\nu$ is the
valuation ring of $\nu$, and $k_\nu$ is the residue field.
\item $C$ is a  curve over a
  number field with Jacobian $J$ and genus $g=g(C)$. A curve $C$ is
  always assumed to be smooth, projective and absolutely irreducible,
  unless explicitly mentioned otherwise.
\item $B_{p,\infty}$ is the quaternion algebra ramified at $p$ and $\infty$, 
and $R$ is a maximal order of $B_{p,\infty}$.
\item For a matrix $T$, $\Tr(T)$ denotes the sum of its diagonal entries, the trace.
\item $\Tr_{K/K_1}$ denotes the trace of a field extension $K/K_1$.
\item For an element of a central simple algebra, $\Nrd$ denotes the reduced norm.
\item $\N_{K/K_1}$ denotes the norm of a field extension $K/K_1$; we use $\N$ when the extension is clear.
\end{itemize}

\subsection*{Acknowledgments}
The authors would like to thank  the Centre International de Rencontres
Math\'ematiques in Luminy for sponsoring the Women in Numbers - Europe (Femmes
en nombre) workshop and for providing a productive and enjoyable
environment for our initial work on this project.  We would especially like to thank the organizers of WINE,
Marie Jos\'e Bertin, Alina Bucur, Brooke Feigon, and Leila Schneps for making the conference and this collaboration possible.
We also thank the referee for the detailed and helpful report.

\section{The Galois group of the Galois closure of a sextic CM-field}
\label{sec:Galois}

Let $K$ be a sextic CM-field, i.e., $K$ is a totally imaginary quadratic
extension of a totally real field $K^+$ with $[K^+:\Q]=3$.
 We denote the Galois closure of $K^+/\Q$ by $L$ and the Galois
 closure of $K/\Q$ by $F$. We write $G=\Gal(F/\Q)$ and
 $G^+=\Gal(L/\Q)$.  The following proposition lists the possibilities
 for $G$.

\begin{proposition}\label{prop:classification}
Let $K$ be a sextic CM-field, and let $G$ be the Galois group of the
Galois closure of $K/\Q$. Then $G$ is one of the following groups:
\begin{enumerate}
\item $C_2\times C_3\simeq C_6$,
\item $C_2\times S_3\simeq D_{12}$,
\item $(C_2)^3\rtimes G^+$ with $G^+\in \{C_3, S_3\}$ acting by permutations on the three copies of $C_2$.
\end{enumerate}
In particular, if $K/\Q$ is Galois, then the Galois group
$G=\Gal(K/\Q)\simeq C_6$ is cyclic.
\end{proposition}

\begin{proof}
This is proved in Section 5.1.1 of \cite{Dodson}, for example.
\end{proof}

In the rest of this section, we sketch the proof of Proposition
\ref{prop:classification}, following Dodson. Since we restrict to the
case of sextic CM-fields, the presentation can be simplified. In the
course of the proof, we also give more details on the structure of the
extensions $F/\Q$ and $K^+/\Q$ in the different cases. In particular,
we show that Case 3 is precisely the case where $K$ does not contain
an imaginary quadratic subfield.

Galois theory implies that we have the following exact sequence of groups:
\[
1\to\Gal(F/L)\to G\to G^+\to 1.
\]

\begin{lemma} \label{lem:extension}
We have
\[
\Gal(F/L)\simeq (C_2)^v, \qquad 1\leq v\leq 3
\]
and
\[
G^+\in \{C_3, S_3\}.
\]
\end{lemma}

\begin{proof}
This lemma is a special case of the proposition in Section 1.1 of
\cite{Dodson}. We give the proof here for convenience.

We first remark that $K=K^+(\sqrt{-\delta})$ for some  totally positive
square-free $\delta\in K^+$. We write $\delta_1:=\delta,
\delta_2, \ldots, \delta_r$ for the $G^+$-conjugates of $\delta$. It
follows that
\[
F=L(\sqrt{-\delta_1}, \ldots, \sqrt{-\delta_r}).
\]
Every element $h\in \Gal(F/L)$ sends $\sqrt{-\delta_i}$ to
$\pm\sqrt{-\delta_i}$. Moreover, $h$ is determined by its action on
these elements. It follows that $\Gal(F/L)\simeq (C_2)^v$ is an
elementary abelian $2$-group.

Since $\delta\in K^+$ it follows that $[\Q(\delta):\Q]$ divides
$3$. We conclude that the number of $G^+$-conjugates of $\delta$
is at most $3$.

The statement on $G^+$ immediately follows from the fact that $[K^+:\Q]=3$.
This proves the lemma.
\end{proof}

\begin{proof}[Proof of Proposition~\ref{prop:classification}]
We start the classification. Note that $\Gal(K/K^+)$ is generated by
complex conjugation. It follows that complex conjugation is also an
element of $G$. This element, which we denote  by $\rho$,  is an element of
the center of $G$.

\bigskip\noindent
\textbf{Case I}: $K/\Q$ Galois.

Since $K/\Q$ is Galois, $G=\Gal(K/\Q)$ is a group of
order $6$, hence either cyclic or $S_3$. Since the Galois closure $L$
of $K^+/\Q$ is a totally real subfield of $K$, it follows that
$K^+=L$.  This implies that $\Gal(K/K^+)$ is a normal subgroup of $G$
which has order $2$. It follows that $G\simeq C_6$ is cyclic.  Note
that $K$ contains the imaginary quadratic subfield $K_1:=K^{C_3}$ and
$K=K_1K^+$. This  corresponds to Case 1 of Proposition
\ref{prop:classification}.

\bigskip\noindent 
\textbf{Case II}: $K/\Q$ is not Galois and $K$
contains an imaginary quadratic field $K_1$.

Since $K$ contains an imaginary quadratic field $K_1$, we have
$F=LK_1$ and $G\simeq C_2\times G^+$. If $G^+\simeq C_3$, then $L=K^+$ and
$K/\Q$ is Galois, which contradicts our assumption. It follows that
$G^+\simeq S_3$ and $G\simeq C_2\times S_3$. This is Case 2 of
Proposition \ref{prop:classification}. We obtain the field
diagram in Figure~\ref{fig:fielddiagram2}.

\begin{figure}[h]
\begin{tikzpicture}[node distance=2cm]
 \node (Q)     at (0,0)             {$\mathbb{Q}$};
 \node (QB)     at (-2,2)             {$K_1$};
 \node (K+)  at (2, 3)   {$K^+$};
 \node (K) [above of=K+]  {$K$};
 \node[right] (L)     at (4,4.5)             {$L$};
 \node (F) [above of=K]  {$F=LK_1 $};
\node[left] at (0.9,1.5) {3};
  \node[left] at (2,4) {2};
\node[right] at (-1,1.2) {2};
\node[left] at (3.5,4.2) {2};
\draw (Q)   -- (K+);
\draw (Q)   -- (QB);
 \draw (K+)  -- (K);
 \draw (K+)  -- (L);
 \draw (K)  -- (F);
 \draw (QB)  -- (F);
 \draw (QB)  -- (K);
 \draw (L)  -- (F);
\end{tikzpicture}
\caption{Field extensions in Case 2}
\label{fig:fielddiagram2}
\end{figure}
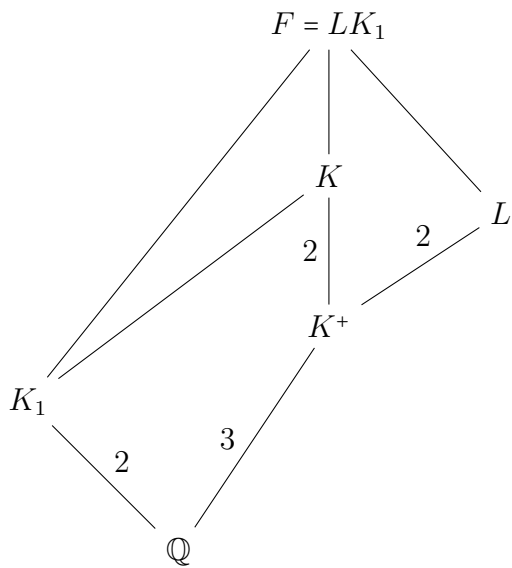

\bigskip\noindent 
\textbf{Case III}: $K/\Q$ is not Galois and $K$
does not contain an imaginary quadratic subfield.

This case corresponds to Case 3 of
Proposition \ref{prop:classification}. In this case the integer $v$
from Lemma \ref{lem:extension} is not equal to $1$, i.e., we have
$v=2$ or $3$. The following claim completes the proof of
Proposition \ref{prop:classification}.

\bigskip\noindent \textbf{Claim}: The case $v=2$ does not occur.
This claim is a special case of the second proposition in Section
5.1.1 of \cite{Dodson}. We give the proof here for completeness.

 Recall that $\rho\in \Gal(F/L)$ denotes complex conjugation and is
 contained in the center of $G$.  Let $\sigma\in G^+$ be an element of
 order $3$.  Then $\sigma$ acts on $\Gal(F/L)=(C_2)^v$ by
 conjugation. This action has two orbits of length $1$, corresponding
 to the identity element and $\rho$. All other orbits have length
 $3$. It follows that
$3\mid (2^v-2)$.
The claim follows.
\end{proof}

Of primary interest to us in the rest of this paper is Case 3 of
Proposition~\ref{prop:classification}, in which $K$ does not contain
an imaginary quadratic subfield.  We have see that $G\simeq (C_2)^3\rtimes
G^+$ with $G^+\in
\{C_3,S_3\}$.
The following diagram describes the field extensions in Case 3.

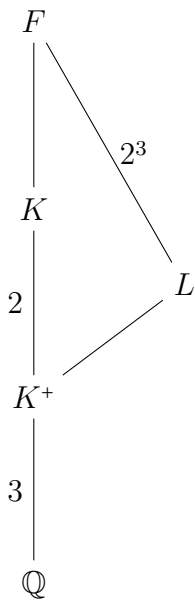
\begin{figure}[h]
\begin{tikzpicture}[node distance=3cm]
 \node (Q)      at (0,0)            {$\mathbb{Q}$};
 \node (K+)  at (0,2.5)   {$K^+$};
 \node (K) at (0,5)  {$K$};
 \node (F) at (0,7.5) {$F$};
  \node (L) at (2,4) {$L$};
\draw (Q)   -- (K+);
 \draw (K+)  -- (K);
\draw (K+)   -- (L);
\draw (L)  -- (F);
\draw (K)   -- (F);
 \node[left] at (0, 1.25)   {$3$};
 \node[left] at (0, 3.75)   {$2$};
\node[right] at (1,5.75)    {$2^3$};
\end{tikzpicture}
\caption{Field extensions in Case 3}
\end{figure}

\section{Primitive CM-types}\label{sec:prim}

Let $K$ be a sextic CM-field. As in Section \ref{sec:Galois}, we write
$K^+$ for the totally real cubic subfield of $K$. The complex
embeddings $K\hookrightarrow
\C$ come in pairs $\{\psi, \rho\circ \psi\}$, where $\rho$ denotes complex
conjugation. Recall that a CM-type $(K, \varphi)$ is a choice of one
embedding from each of these pairs.  The goal of this section is to
determine the primitive CM-types. We start by recalling the definition
from \cite{Milne}, Section 1.1. For examples we refer to
Section \ref{sec:exa}.

\begin{defn}
Let $(K,\varphi)$ and
$(K_1,\varphi_1)$ be CM-types. We say that $(K,\varphi)$ is {\em
induced} from $(K_1, \varphi_1)$ if $K_1$ is a subfield of $K$
and the restriction of $\varphi$ to $K_1$ coincides with
$\varphi_1$. A CM-type is called {\em primitive} if it is not induced
from a CM-type on any proper CM-subfield of $K$.
\end{defn}

Let $A$ be an abelian variety and let $K$ be a CM-field with
$[K:\Q]=2\dim(A)$. We say that $A$ has {\em complex multiplication
  (CM) by} $K$ if the endomorphism algebra
$\End^0(A)=\End(A)\otimes\Q$ contains $K$.  We say that a curve $C$
has {\em CM by} $K$ if its Jacobian has CM by $K$. We say that $A$ (or
$C$) has CM if there exists a CM-field $K$ such that $A$ (or $C$) has
CM by $K$. If $\End(A)$ is an order ${\mathcal O}$ in a CM-field $K$
with $[K:\Q]=2\dim(A)$, we say that $A$ has {\em CM by} ${\mathcal
  O}$.

The following theorem gives a geometric interpretation of what it
means for the CM-type of a CM-abelian variety to be primitive in
characteristic zero.   For
convenience, we say that an abelian variety $A$ defined over a field
$M$ is simple if it is absolutely simple, meaning that
$A\otimes_M \overline{M}$ is not isogenous to a product of abelian
varieties of lower dimension. Similarly, we say that two abelian
varieties $A_1, A_2$ defined over $M$ are isogenous if there exists an
isogeny $\varphi:A_1\to A_2$ defined over the algebraic closure of
$M$.

\begin{theorem}\label{thm: primsimple}
 Let $A$ be an abelian variety defined over a field of characteristic
 zero.  Suppose that $A$ has CM with CM-type $(K,\varphi)$.  Then
 the CM-type $(K, \varphi)$ is primitive if and only if the abelian
 variety $A$ is simple.
\end{theorem}

\begin{proof} This is proved in  Theorem 3.5 of Chapter 1
 of \cite{Lang}. See also Remark 1.5.4.2 of \cite{ChaiConradOort}.
\end{proof}

We refer to Section 1.5.5 of \cite{ChaiConradOort} for an explanation
of why we need to assume that $A$ is defined over a field of
characteristic zero in Theorem \ref{thm: primsimple}.

The following result gives a useful criterion for determining
whether a given CM-type is primitive. For a proof, we refer to
Theorem 3.6 of Chapter 1 of \cite{Lang}. For a CM-type $(K, \varphi)$ and
$h\in \Aut(K)$, we write
\[
\varphi h=\{\varphi_i\circ h\mid \varphi_i\in \varphi\}.
\]

\begin{proposition}\label{prop: Lang} Let $(K, \varphi)$ be a CM-type. We write
 $(F, \Phi)$ for the induced CM-type of the Galois closure of
  $K/\Q$. Let
\[
H_\Phi=\{h\in G=\Gal(F/\Q)\mid \Phi h=\Phi\}.
\]
Then $(K, \varphi)$ is primitive if and only if
\[
K=F^{H_\Phi}.
\]
\end{proposition}

We now determine the primitive sextic CM-types in each of the cases of
Proposition \ref{prop:classification}. We first consider Case
$3$. Recall that in the proof of Proposition \ref{prop:classification}
we showed that Case 3 is precisely the case where $K$ does not contain
an imaginary quadratic subfield.

\begin{cor}\label{cor: case3}
Suppose that we are in Case 3 of
Proposition \ref{prop:classification}, i.e., $K$ does not contain an
imaginary quadratic field. Then every CM-type $(K, \varphi)$ is
primitive.
\end{cor}

\begin{proof} Suppose  that $(K,\varphi)$ is not primitive.
 Then $K$ contains a proper CM-subfield $K_1$. Since $K$ is sextic,
 $K_1$ is an imaginary quadratic field. This yields a contradiction.
\end{proof}

\subsection{Primitive types in Case 1}
\label{subsec:Primitive types in Case 1}

We now consider Case 1 from
Proposition~\ref{prop:classification}. This is the case in which
$K/\Q$ is Galois, with Galois group $G\simeq C_6$. We choose a
generator $\sigma$ of $G$. Note that complex conjugation corresponds
to $\sigma^3$. Up to replacing $\varphi$ by its complex conjugate,
every CM-type $(K, \varphi)$ may be written as
\[
\varphi_{a,b}=\{1, \sigma^a, \sigma^b\}, \qquad 0<a,b<6, \quad a\equiv
1\pmod{3},\ b\equiv 2\pmod{3}.
\]
We find $4$ cases:
\[
\{a,b\}\in \{\{1,2\}, \{1,5\}, \{4,2\}, \{4,5\}\}.
\]
Note that changing the generator $\sigma$ of $G$ to $\sigma^{-1}$
changes $\{4,5\}$ to $\{1,2\}$, therefore we do not have to consider
the choice $\{4,5\}$.

We write $H_{a,b}$ for the subgroup fixing the CM-type as in
Proposition \ref{prop: Lang}. Then $H_{1,2}=H_{1,5}=\{1\}$ and
$H_{4,2}=\langle \sigma^2\rangle\simeq C_3$. Note that
$K_1:=K^{H_{4,2}}$ is the imaginary quadratic subfield of
$K$, which is a CM-field. We conclude that $\varphi_{4,2}$ is
induced from $K_1$, and hence imprimitive. The other CM-types are primitive.

\subsection{Primitive types in Case 2}

 We now consider Case 2 from
Proposition \ref{prop:classification}. We refer to Section
\ref{sec:Galois} for a description of the fields involved.  Recall
 that $K=K_1K^+$. Therefore, an embedding $\psi:K\hookrightarrow \C$
corresponds to an ordered pair $(\psi_1, \psi^+)$, where
$\psi_1:K_1\hookrightarrow \C$ is an embedding of $K_1$ and
$\psi^+:K^+\hookrightarrow \C$ is an embedding of $K^+$. Since $K^+$
is totally real, the image of $\psi^+$ is contained in $\R$. We denote the three
possible complex embeddings of $K^+$ by $\chi_i$ for $i=1,2,3$.  We
fix a complex embedding of $K_1$ and denote it by $1$. We denote the other
complex embedding of $K_1$ by $-1$.

A CM-type
 $(K, \varphi)$ consists of a triple of these ordered pairs in which no two of the pairs are complex conjugates.
Since $\Gal(K_1/\Q)$ is generated by complex conjugation, we simply
 choose one of the two complex
 embeddings of $K_1$ for each embedding $\chi_i$ of $K^+$. This means
 that we may write
\[
\varphi=\{(\epsilon_i, \chi_i)\mid i=1,2,3\}, \quad \epsilon_i\in \{\pm 1\}.
\]
Identifying $\varphi$ with its complex conjugate yields four different
CM-types.

We determine the imprimitive types. The only CM-field properly
contained in $K$ is the imaginary quadratic field $K_1$. The
restriction of the embedding $(\epsilon_i, \chi_i)$ to $K_1$ is just
$\epsilon_i$. Therefore, the CM-type
$\varphi=\{(\epsilon_i, \chi_i)\}$ is imprimitive if and only if
$\epsilon_i$ is independent of $i$.  We conclude that there is a
unique imprimitive CM-type. The other three are primitive.

\subsection{Examples of CM-types}
\label{subsec:CMexa}
We give examples of CM-types illustrating each of the three cases of
Proposition \ref{prop:classification}.
\begin{example}[$K/\Q$ is Galois with Galois group $G \simeq C_6$.]
\label{eg:zeta7}
Let $K$ be $\Q(\zeta_7)$ where $\zeta_7$ is a primitive seventh root
of unity. The maximal totally real subfield of $K$ is
$K^+=\Q(\zeta_7+\zeta_7^{-1})$, which has degree three over $\Q$ (the
minimal polynomial of $\zeta_7+\zeta_7^{-1}$ over $\Q$ is $x^3 + x^2 -
2x - 1$). The field $K$ is a totally imaginary quadratic extension of
$K^+$.

The automorphism $\sigma$ which maps $\zeta_7$ to $\zeta_7^5$
generates $ \Gal(K/\Q)$.  The fixed field of $\langle\sigma^2\rangle$
is $\Q(\zeta_7^4 + \zeta_7^2 + \zeta_7)=\Q(\sqrt{-7})$. This  is the
unique imaginary quadratic extension of $\Q$ contained in
$\Q(\zeta_7)$. Therefore, the only imprimitive CM-type admitted by $K$ is
$\varphi_{2,4}=\{1,\sigma^2,\sigma^4\}$; the
CM-types $\varphi_{a,b}=\{1, \sigma^a,\sigma^b\}$ for $\{a,b\} \neq \{4,2\}$ with $a\equiv 1\pmod{3},\ b\equiv 2\pmod{3}$ are all primitive.
\end{example}

The following examples have been taken from the database of Kl\"uners
and Malle (\cite{KluenersMalle}).

\begin{example}[The Galois closure of $K/\Q$ is $D_{12}$]
Let $K$ be the sextic field obtained by adjoining a root of the
irreducible polynomial $f(x)=x^6 - 3x^5 + x^4 + 10x^2 - 9x + 3$. Then
$K$ is a totally imaginary quadratic extension of the totally real
cubic field $K^+=\Q(\alpha)$ where the minimal polynomial of $\alpha$
is $g(x)=x^3 - 7x^2 + 12x - 3$. The Galois closure $F$ of $K/\Q$ is
the compositum of the Galois closure of $K^+$ with the unique
imaginary quadratic subfield $K_1$ of $K$, given by the minimal
polynomial $x^2+3x+3$.  The Galois group of $F$ is isomorphic to
$S_3 \times C_2\simeq D_{12}$.  Denote the roots of $g(x)$ by
$\alpha_1:=\alpha, \alpha_2,\alpha_3$.

 Let $\chi_i: \alpha_1 \mapsto \alpha_i$ denote the three real
 embeddings of $K^+$ and $\pm 1$ denote the two complex embeddings of
 $K_1$.  Then the CM-type $\varphi=\{(1,\chi_1),(1,
 \chi_2),(1,\chi_3)\}$ of $K$ is imprimitive, since its restriction to
 the quadratic imaginary subfield $K_1$ is also a CM-type.  The remaining
 three CM-types of $K$ are primitive. For clarity, the
 primitive CM-types are as follows: $\{(1,\chi_1), (-1,\chi_2), (-1,\chi_3)\},
 \{(1,\chi_1),(1,\chi_2),(-1,\chi_3)\},
 \{(1,\chi_1),(-1,\chi_2),(1,\chi_3)\}$.
\end{example}

\begin{example}[The Galois closure of $K/\Q$ is $(C_2)^3 \rtimes C_3$]
  Let $K = \Q(\beta)$ be the degree $6$ extension of $\Q$ where the
minimal polynomial of $\beta$ is $f(x)=x^6 - 2x^5 + 5x^4 - 7x^3 +
10x^2 - 8x+8$. Let $F$ be the Galois closure of $K$.  Then
$\Gal(F/\Q)$ is $(C_2)^3
\rtimes C_3$.
 Moreover, $K$ is a CM-field since $K$ is a totally imaginary
 quadratic extension of $K^+=\Q(\alpha)$ where the minimal polynomial of
 $\alpha$ over $\Q$ is $g(x)=x^3 - 7x^2 + 14x - 7$.  Note that $K$ contains
no quadratic subfield, hence every CM-type is primitive.  \end{example}

\subsection{Comparison with the genus $2$ case}\label{sec:g=2}

The following proposition characterizes primitive CM-types for quartic
CM-fields. 

\begin{proposition}\label{prop:g=2} Let $K$ be a quartic CM-field. The 
following are equivalent.
\begin{enumerate}
\item The CM-type is primitive.
\item The CM-field $K$ does not contain an imaginary quadratic
subfield.
\end{enumerate}
\end{proposition}

\begin{proof}
 We recall the argument from Example 8.4.(2) of \cite{Shimura} in
 which we find a classification of the possible Galois groups of quartic 
 CM-fields $K$  together with the possible CM-types.
It follows from this classification that if $K$ contains a proper
 CM-subfield $K_1\neq \Q$ then $K/\Q$ is Galois with Galois group
 $G\simeq C_2\times C_2$. Moreover, in this case all CM-types are
 imprimitive. Namely, denoting again complex conjugation by $\rho$, we may
 write $G=\{1, \rho, \sigma, \rho\sigma\}$. Then the possible CM-types
 are $\{1, \sigma\}$ and $\{1, \rho\sigma\}$, which are fixed by
 $\langle\sigma\rangle$ and $\langle \rho\sigma\rangle$,
 respectively. Therefore, the statement follows from Proposition \ref{prop:
   Lang}.
\end{proof}

Proposition \ref{prop:g=2} explains why Goren and Lauter (\cite{GorenLauter06},
\cite{GorenLauter07}) restrict to the case where the quartic CM-field does not
 contain an imaginary quadratic subfield. For quartic CM-fields, this
 is equivalent to requiring that the CM-type is primitive. However, as
 we have seen in our discussion of the primitive types in Cases 1 and
 2 of Proposition \ref{prop:classification}, these two properties are
 not equivalent for sextic CM-fields.

We give two concrete examples of genus $2$ curves with CM to
illustrate Proposition~\ref{prop:g=2}.  These are similar to the
genus $3$ examples given in Section \ref{sec:cyclicexa}. We consider
two smooth projective curves defined by the following affine equations
\[
\begin{split}
D_1:&\qquad   y^5=x(x-1),\\
D_2:&\qquad   y^8=x(x-1)^4.
\end{split}
\]
One easily verifies that both curves have
genus $2$.

The curve $D_1$ has CM by $K_1:=\Q(\zeta_5)$ with CM-type $(1,2)$ in
the notation of Section \ref{sec:cyclicexa}.  The Galois group of
$K_1/\Q$ is cyclic of order $4$, hence its unique subgroup of order
$2$ is generated by complex conjugation, which cannot fix the
CM-type. Indeed, the Jacobian of $D_1$ is simple. In the genus $2$
case, all CM-types of a cyclic CM-field are primitive. We have already
seen that this does not hold in general for genus $g\geq 3$.

The curve $D_2$ has CM by $K_2:=\Q(\zeta_8)$. The corresponding
Galois group is isomorphic to $C_2\times C_2$, hence the CM-type is
imprimitive. Indeed, the CM-type is $(1,3)$ which is fixed by $\langle
3\rangle\subset (\Z/8\Z)^\ast.$ The CM-type $(1,3)$ is induced from the CM-type of the
elliptic curve $E:=D_2/\langle\tau\rangle$, where $\tau(x,y)=(1/x,
y^3/x(x-1))$ is an automorphism of order $4$.

\section{Reduction of CM-curves and their Jacobians} \label{sec:jacobian3}

Our main result (Theorem \ref{thm:bnd_on_p}) deals with curves $C$ of
genus $3$ defined over some number field whose Jacobians have CM by a
sextic CM-field $K$.  In this section, we describe the
possibilities for the reduction of these curves and their Jacobians to
characteristic $p>0$.

\subsection{The theorem of Serre--Tate}

  Let $C$ be a curve of
genus $g\geq 2$ defined over a number field $M$, and let $J:=\Jac(C)$
be its Jacobian. In the course of our arguments, we allow ourselves to
replace $M$ by a finite extension,
 which we still denote by $M$.  Let
$\nu$ be a finite place of $M$. We write ${\mathcal O}_\nu$ for the
valuation ring of $\nu$ and $k_\nu$ for its residue field.  We write
$\overline{k_{\nu}}$ for an algebraic closure of $k_\nu$.

Recall that the abelian variety $J$ has {\em good reduction} at $\nu$
if there exists an abelian scheme ${\mathcal J}$ over ${\mathcal
  O}_\nu$ with ${\mathcal J}\otimes_{{\mathcal O}_\nu} M\simeq J$.
This implies that the reduction $\oJ:={\mathcal
  J}\otimes_{{\mathcal O}_\nu} \overline{k_\nu}$ is an abelian
variety.  We say that $J$ has {\em potentially good reduction} at
$\nu$ if there exists a finite extension $M'/M$ and an extension
$\nu'$ of $\nu$ such that $J\otimes_M M'$ has good reduction at
$\nu'$.

The following theorem is Theorem 6 of \cite{SerreTate}.

\begin{theorem}(Serre--Tate)\label{thm:SerreTate}
Let $J$ be an abelian variety with CM defined over a number field
$M$. Let $\nu$ be a finite place of $M$.  Then $J$ has potentially
good reduction at $\nu$.
\end{theorem}

 Since there
  are at most finitely places where $J$ does not have good reduction,
  there exists a finite extension of $M$ over which $J$ has good
  reduction everywhere.

\subsection{Reduction of genus $3$ curves with CM}\label{sec:curvered}

We now describe the restrictions imposed by Theorem \ref{thm:SerreTate}
on the reduction of the curve $C$. 

 Recall that $C$ is a curve of genus $g(C)\geq 2$ defined over a
 number field $M$. We say that $C$ has {\em good reduction} at a
 finite place $\nu$ of $M$ if there exists a model $\mathcal{C}$ over
 ${\mathcal O}_\nu$ with $\mathcal{C}\otimes_{{\mathcal O}_\nu}
 M\simeq C$ such that the reduction
 $\oC:=\mathcal{C}\otimes_{{\mathcal O}_\nu} \overline{k_\nu}$ is
 smooth. Similarly, $C$ has {\em potentially good reduction} at $\nu$
 if it has good reduction over a finite extension of $M$.

We say that $C$ has {\em semistable reduction} at $\nu$ if there
exists a model $\mathcal{C}$ over ${\mathcal O}_\nu$ with
$\mathcal{C}\otimes_{{\mathcal O}_\nu} M\simeq C$ such that the
reduction $\oC$ is semistable. This means that $\oC$ is reduced and
has at most ordinary double points as singularities. The corresponding
model ${\mathcal C}={\mathcal C}_\nu$ is called a {\em semistable
  model} of $C$ at $\nu$.  The Stable Reduction Theorem
(\cite{DeligneMumford}, Corollary 2.7) states that every curve $C$
admits a semistable model at $\nu$ after replacing $M$ by a finite
extension. Since we assume that $g(C)\geq 2$, there exists a unique
minimal semistable model, which is called the {\em stable model} at $\nu$. Its
special fiber $\oC$ is called the {\em stable reduction} of $C$ at
$\nu$. The minimality of the stable model implies that $C$ has
potentially good reduction if and only if the stable reduction $\oC$
is smooth. If the finite place $\nu$ is fixed, we usually omit it.

We now turn to our situation of interest, namely that of a genus $3$
curve whose Jacobian has CM by a sextic CM-field. The following
proposition is a consequence of Theorem \ref{thm:SerreTate}.

 We say that $C$ has {\sl bad reduction} at $\nu$ if it does not have
potentially good reduction at $\nu$.  This is equivalent to the stable
reduction $\oC$ having singularities. We say that the reduction
$\bar{C}$ of $C$ is {\em tree-like} if the intersection graph of the
irreducible components of $\oC$ is a tree.
Note that we always consider the reduction $\oJ$ (resp.\ $\oC$) as an
abelian variety (resp.\ curve) defined over the algebraically closed
field $\overline{k_\nu}$ for convenience.

\begin{proposition}\label{prop:SerreTate}
 Let $C$ be a curve of genus $3$ defined over a number field $M$ such
 that its Jacobian $J=\Jac(C)$ has CM.
Let $\nu$ be place of $M$ where $C$ has bad reduction.
Then
\begin{itemize}
\item[(a)] the stable reduction reduction $\oC$ of $C$ is tree-like,\\ and
\item[(b)] the reduction $\oJ$ of $J$ is the product of the Jacobians
  of the irreducible components of $\oC$ (as polarized abelian varieties).
\end{itemize}
\end{proposition}

\begin{proof}
Let $\nu$ be a finite place of $M$. After replacing $M$ by a finite
extension and choosing an extension of $\nu$, we may assume that $C$
has stable reduction at $\nu$.  Let $\mathcal{C}$ be the stable model
of $C$.  Set $S=\Spec({\mathcal O}_\nu)$, and define $\Pic^0(\mathcal C/S)$
to be the identity component of the Picard variety. Since the stable reduction
$\oC$ of $C$ is reduced, Theorem 1 in Section 9.5 of \cite{BLR} states
that $\Pic^0({\mathcal C}/S)$ is a N\'eron model of $J$.

Theorem \ref{thm:SerreTate} implies that $J$ has potentially good
reduction, i.e.,  there exists an abelian variety ${\mathcal J}$ over
$S$ with generic fiber $J$. Proposition 8 of Section 1.2 in \cite{BLR}
shows that ${\mathcal J}/S$ is a N\'eron model. Since two different
N\'eron models are canonically isomorphic, it follows that
$\Pic^0({\mathcal C}/S)\simeq_S {\mathcal J}$. In particular, it
follows that the special fiber $\Pic^0({\mathcal
  C}/S)\otimes_{{\mathcal O}_\nu} \overline{k_\nu}\simeq \Pic^0(\oC)$ is an abelian
variety.

Example 8 of Section 9.2 in \cite{BLR} shows that $\Pic^0(\oC)$ is
given by an exact sequence
\begin{equation}
\label{eq:extension}
1\to T\to \Pic^0(\oC)\to B:=\prod_i \Jac(\widetilde{C}_i)\to 1,
\end{equation}
 where $B$ is an abelian variety and $T$ is a torus.  The product
on the right-hand side is taken over the irreducible components of
$\oC$. We denote the normalization of an irreducible component ${C}_i$
of $\oC$  by $\widetilde{C}_i$.
The torus $T$ satisfies
\[
T\simeq \G^t_{m, {\overline{k_\nu}}}
\]
for some $t\geq 0$.  The torus $\G_m$ is not compact, and hence not an
abelian variety. Since $\Pic^0(\oC)$ is an abelian variety, the exact
sequence \eqref{eq:extension} implies that $t=0$, i.e., $\Pic^0(\oC)$
contains no torus.  By Corollary 12.b of \cite{BLR}, this means that
the intersection graph of the irreducible components of $\oC$ is a
tree.
 Both statements of the proposition follow from this.
\end{proof}

The corollary below follows immediately  
from Proposition
\ref{prop:SerreTate}.  In Section \ref{sec:exa}, we give examples of
each of the cases.

\begin{cor}\label{cor:SerreTate} Let $C$ be a genus $3$ curve 
with CM defined over a number field $M$, and let $\nu$ be a
finite place of $M$.  One of the following three possibilities holds
for the irreducible components of $\oC$ of positive genus:
\begin{itemize}
\item[(i)] (good reduction) $\oC$ is a smooth curve of genus $3$,
\item[(ii)] $\oC$ has three irreducible components of genus $1$,
\item[(iii)] $\oC$ has an irreducible component of genus $1$ and one
  of genus $2$.
\end{itemize}
\end{cor}

Note that the stable reduction $\oC$ may contain irreducible
components of genus $0$. This happens for the stable
reduction $\oC_1$ to characteristic $3$ of the curve $C_1$ from
Lemma \ref{lem: C1}, for example.  One may show that $\oC_1$ has four irreducible
components: one of genus $0$ and three of genus $1$. The three
elliptic curves each intersect the genus $0$ curve in one point but do
not intersect each other.  Since the irreducible components of genus $0$
do not contribute to the Jacobian, we have not listed them in
Corollary \ref{cor:SerreTate}.

\begin{remark}\label{rem:SerreTate}
Let $C$ be a curve of genus $3$ with CM, defined over a number field $M$. Suppose that $C$ has bad reduction at a finite place $\nu$ of $M$.
In Case (ii) of Corollary~\ref{cor:SerreTate}, the reduction $\oC$ of
$C$ contains three irreducible components $E_i$ of genus
$1$. Proposition~\ref{prop:SerreTate} implies that
\[
\oJ\simeq E_1\times E_2\times E_3
\]
as polarized abelian varieties, i.e., the polarization on $\oJ$ is the
product polarization.

In Case (iii) of Corollary \ref{cor:SerreTate}, $\oC$ contains an
irreducible component $E$ of genus $1$ and an irreducible component
$D$ of genus $2$.  In this case, we have
\[
\oJ\simeq E\times \Jac(D)
\]
and the polarization on $\oJ$ is induced by $E\times\{0\} +\{0\}\times
D\hookrightarrow \oJ$. We  show below that in this case $\oJ$ is
still isogenous to a product of elliptic curves (Theorem \ref{thm:
  Lang2}). However, it is {\bf not} true that the polarization of
$\oJ$ is induced by polarization on the three elliptic curves as we
had in Case (ii).

Even in the case where $C$ has good reduction (Case (i) of
Corollary \ref{cor:SerreTate}), the reduction $\oJ$ of the Jacobian
need not be simple even if $J$ is. In this case, the polarization of
$\oJ$ is induced by the embedding of $\oC$ in its
Jacobian and hence is not a product polarization.
\end{remark}

 The following theorem  is  a generalization of Theorem \ref{thm:
  primsimple} to positive characteristic. 

\begin{theorem}\label{thm: Lang2}
Let $J$ be an abelian variety of dimension $3$ with CM,
defined over a number field $M$.  Suppose that the reduction $\oJ$ of
$J$ at a finite place of $M$ is not simple.  Then $\oJ$ is isogenous to the
product of three copies of the same elliptic curve $E$.
\end{theorem}

\begin{proof}
The result is essentially a special case of Theorem 1.3.1.1
of \cite{ChaiConradOort}. For the convenience of the reader we sketch
a direct proof in our situation.

By assumption, $J$ has CM by the sextic CM-field $K$. This implies that
we have an embedding
\[
 K\hookrightarrow \End^0(\oJ).
\]
Decompose $\oJ$ into isotypic components: $\oJ \sim \prod_{i}
{A_i^{n_i}}$ where the $A_i$ are simple and $A_i\not\sim A_j$ for
$i\neq j$. Since $\oJ$ is not simple by assumption, for dimension reasons there exists $j$ such that
$A_j=E$ is an elliptic curve. We have
$K\hookrightarrow \End^0(\oJ)=\prod_{i}
{M_{n_i}(\End^0(A_i))}$. Projecting this ring homomorphism on the
$j$th factor gives an injection $K\hookrightarrow
M_{n_j}(\End^0(E))$. A dimension argument shows that $n_j=3$ and
therefore $\oJ\sim E^3$.
\end{proof}

\begin{proposition}\label{prop: K_1 in K}
Let $C$ be a genus $3$ curve with CM, defined over a number field $M$.
Suppose that $C$ has bad reduction at a finite place $\nu$
of $M$. Then  the reduction $\oJ$ of the Jacobian $J$ of $C$ is
supersingular or $K$ contains an imaginary quadratic field $K_1$.
\end{proposition}

\begin{proof}
Let $C$ and $\oJ$ be as in the statement of the proposition.  Since
$C$ has bad reduction at $\nu$, Corollary \ref{cor:SerreTate} shows
that $\oC$ has an irreducible component $E_1$ of genus $1$. It follows
that we may regard $E_1$ as abelian subvariety of $\oJ$. (This is
slightly weaker than the statement in Remark \ref{rem:SerreTate}.)  In
particular, $\oJ$ is not simple.  Theorem \ref{thm: Lang2} implies
therefore that $\oJ$ is isogenous to the product of three copies of an
elliptic curve $E$.  Note that $\oJ$ is supersingular if and only if
$E$ is.

We  assume that $E$ is ordinary. Since $E$ may be defined over a finite
field, it has CM and $K_1:=\End^0(E)$ is an imaginary quadratic field
contained in the center of $\End^0(E^3)=M_3(K_1)$. Since $\oJ$
is isogenous to $E^3$, we obtain an embedding
\[
K=\End^0(J)\hookrightarrow \End^0(\oJ)\simeq \End^0(E^3)=M_3(K_1).
\]
Theorem 1.3.1.1 of \cite{ChaiConradOort} states that $K$ is its own
centralizer in $M_3(K_1)$. Since the center of $M_3(K_1)$ is $K_1$, we
conclude that $K_1$ is contained in $K$ and the result follows.
\end{proof}

The following corollary summarizes the results so far in the
case that the CM-field $K$ does not contain an imaginary quadratic
subfield $K_1$.

\begin{cor}\label{cor:K_1 in K}
Let $C$ be a genus $3$ curve with CM by $K$, defined over a number field $M$. 
Suppose that $K$
does not contain an imaginary quadratic subfield. Then  the
following hold:
\begin{itemize}
\item[(a)] the CM-type $(K, \varphi)$ of $J$ is primitive, and $J$ is
  absolutely simple,
\item[(b)] 
 if $C$ has bad reduction at a finite place $\nu$, then the reduction of $J$ at $\nu$ is
  supersingular.
\end{itemize}
\end{cor}

\begin{proof}
  Part (a) follows from Corollary \ref{cor: case3} and
  Theorem \ref{thm: primsimple}. 
Part (b) follows from Proposition \ref{prop: K_1 in K}.
\end{proof}

\subsection{Polarizations and the Rosati involution}\label{sec:polarization}

In the rest of this section, we recall some results on the Rosati
involution following Sections 20 and 21 of \cite{MAV} and Section 17
of \cite{MilneAV}.  For precise definitions and more details, we refer
to these sources. Let $A$ be an abelian variety and $\lambda:A\to
A^\vee$ be a polarization associated with an ample line bundle
${\mathcal L}$ on $A$. The polarization $\lambda$ is an isogeny and
therefore has an inverse $\frac{1}{\deg \lambda} \lambda^{\vee}=
\lambda^{-1}\in \Hom(A^\vee, A)\otimes_\Z \Q$.

The Rosati
involution on $\End^0(A)=\End(A)\otimes \Q$ is defined by
\[
f\mapsto f^\ast=\lambda^{-1}\circ f^\vee\circ \lambda.
\]
It satisfies
\[
(f+g)^\ast=f^\ast+g^\ast, \quad
(f g)^\ast=g^\ast f^\ast, \quad a^\ast=a
\] 
for $f, g\in \End^0(A)$ and $a\in \Q$.  
In the case where $\lambda$ is a principal polarization, i.e.,
$\deg(\lambda)=1$, the Rosati involution acts as an involution on
$\End(A)$. This is because $\lambda^{-1}$ is in $\Hom(A^\vee, A)$ and not
just in $\Hom(A^\vee, A)\otimes_\Z \Q$. The natural polarization on a
Jacobian is a principal polarization.

The Rosati involution is a positive involution (Theorem 1 of Section
21 in \cite{MAV}). This means that
\[
(f, g)\mapsto \Tr(f\cdot g^\ast), \qquad \End^0(A)\to \Q
\]
defines a positive definite quadratic form on $\End^0(A)$. (We refer
to Section 21 of \cite{MAV} for the precise definition of the trace.)
In the case that $A=E$ is an elliptic curve, we choose the
polarization $\lambda$ defined as
\[
\lambda:E\to \Pic^0(E),\quad P\mapsto [P]-[O].
\]
The corresponding Rosati involution sends an isogeny
 $f$ to its dual isogeny $f^\vee$ and
 $\Tr(f\cdot f^\vee)$ is $\deg(f)$, the degree of the
 endomorphism $f$.

\begin{proposition}\label{prop: Rosati}
Let $A$ be a simple abelian variety defined over a field of
characteristic zero with principal polarization $\lambda$. Assume that
$A$ has CM by a field $K$. Then the Rosati involution associated
with $\lambda$ induces complex conjugation on the CM-field $K$.
\end{proposition}

\begin{proof}
Since $A$ is simple, the endomorphism algebra $\End^0(A)$
equals $K$ and the proposition is proved, for example, in Lemma 1.3.5.4 of
\cite{ChaiConradOort}.
\end{proof}

\begin{remark}\label{rem: Rosati} Let $A$ be a simple abelian variety with 
$\End^0(A)=K$ as in the statement of Proposition \ref{prop:
 Rosati}. Let $M$ be a number field over which $A$ can be defined, and
 let $\mathfrak{p}$ be a prime of $M$ at which $A$ has good
 reduction. Write $\overline{A}$ for the reduction.  We obtain an
 embedding
\[
K\hookrightarrow \End^0(\overline{A}).
\]
The Rosati involution on $\End^0(\overline{A})$ is an extension of the
Rosati involution on $\End^0(A)=K$, which is complex conjugation by
Proposition \ref{prop: Rosati}.
\end{remark}

The following proposition was used in the proofs of~\cite{GorenLauter07}
but not stated there explicitly. 

\begin{proposition}\label{prop: Rosati2}
 Suppose that $A=E^n$ is a product of elliptic
 curves as polarized abelian varieties. Then the Rosati involution acts as
\[
M_n(\End(E))\to
M_n(\End(E)), \qquad 
(f_{i,j})\mapsto (f_{j,i}^\vee).
\]
\end{proposition}

\begin{proof}
The result is well known to the experts. We sketch the argument. The
proof we present here is a variant of the proof of Proposition
11.28 (ii) of \cite{GeerMoonen}.

Let $A=E^n$ be as in the statement of the lemma, and write $p_i:A\to E$
for the projection on the $i$th coordinate. Then any line bundle
${\mathcal L}$ on $A$ satisfies ${\mathcal L}=p_1^\ast {\mathcal
L}_1\otimes \cdots \otimes p_n^\ast{\mathcal L}_n$ for suitable line
bundles ${\mathcal L}_i$ on $E$.  

Consider the natural map $A^\vee=\Pic^0(A)\to
(\Pic^0(E))^n=(E^\vee)^n$ which sends a line bundle ${\mathcal
L}\in \Pic^0(A)$ to the $n$-tuple $({\mathcal
L}|_{E_i})_i\in (\Pic^0(E))^n$ of the restrictions of ${\mathcal L}$ to
the $i$th copy $E_i:=(\cdots, 0, E, 0, \cdots)$ of $E$. One shows that
this map is an isomorphism (Exercise 6.2 of \cite{GeerMoonen}). The
product polarization $\lambda_A:A\to A^\vee=(E^\vee)^n$ is induced by
the natural polarization $\lambda:E\to \Pic^0(E)$ on $E$. In
particular, it is also a principal polarization.

 Using this identification, it suffices to prove the proposition in
the case that $f\in \End(A)$ corresponds to a $n\times n$ matrix with
an endomorphism $\alpha\in \End(E)$ as $(j,i)$th component and zeros
everywhere else.  The endomorphism $f^\vee:A^\vee\to A^\vee$ induced by $f$
sends a line bundle ${\mathcal L}$ on $A$ to $p_i^\ast(\alpha^\ast
{\mathcal L}_i)$.  We conclude that the dual isogeny $f^\ast:A\to A=E^n$
corresponds to the matrix with the dual isogeny $\alpha^\vee$ in the
$(i,j)$th coordinate and zeros elsewhere. This proves the proposition.
\end{proof}

\section{Examples}\label{sec:exa}
 In this section we discuss some examples of genus $3$ curves
with CM.

\subsection{Cyclic covers}\label{sec:cyclicexa}
The first type of examples we consider are $N$-cyclic covers of the projective
line branched at exactly three points, see also Sections 1.6 and 1.7 of
Chapter 1 of \cite{Lang}. More precisely, let $C$ be
a smooth projective curve defined over a field of characteristic zero
which admits a Galois cover $\pi:C\to \PP^1$ whose Galois group is
cyclic of order $N$ such that $\pi$ is  branched exactly at three points. We
may assume the three branch points to be $0, 1,\infty\in
\PP^1$. 

Kummer theory implies the existence of integers $0<a_1, a_2<N$
with $\gcd(N, a_1, a_2)=1$ such that the extension of function field
corresponding to $\pi$ is
\[
\Q(x)\subset \Q(x)[y]/(y^N-x^{a_1}(x-1)^{a_2}).
\]
The Galois group of $\pi$ is generated by $\alpha(x,y)=(x, \zeta_N y)$,
where $\zeta_N$ is a primitive $N$th root of unity.

Define $0<a_3<N$ by $a_1+a_2+a_3\equiv 0\pmod{N}$. Then a chart at
$\infty$ may be given by
\[
w^N=z^{a_3}(z-1)^{a_2}, 
\]
where $z=1/x$.  The condition that $\pi$ is branched at $\infty$
is therefore equivalent to $a_3\equiv -(a_1+a_2)\not\equiv 0\pmod{N}$. 
The Riemann--Hurwitz formula shows that
\[
2g(C)-2=-2N+\sum_{i=1}^3(N-\gcd(N, a_i)).
\]

In Lemma \ref{lem:exa} below, we show that the endomorphism ring of
$\Jac(C)$ contains $\Q(\zeta_N)$. Therefore $\Jac(C)$ has CM by
$\Q(\zeta_N)$ if and only if $2g(C)=\varphi(N)$, where $\varphi$
denotes Euler's totient function. For example, this condition is satisfied
if $N$ is an odd prime. This case is discussed by Lang (Section 1.7
of Chapter 1 of \cite{Lang}).

The condition $2g(C)=\varphi(N)$ is satisfied for exactly three curves
$C_i$, up to isomorphism. Kummer theory implies that two tuples $(N,
a_1, a_2, a_3)$ and $(M, b_1, b_2, b_3)$ define isomorphic curves if
and only if $N=M$ and there exists an integer $c$ with $\gcd(c, N)=1$
and a permutation $\sigma\in S_3$ such that $b_i\equiv c
a_{\sigma(i)}\pmod{N}$ for all $i$. This is similar to the argument in Section 1.7
of Chapter 1 of \cite{Lang}.

The three curves satisfying this property are: 
\[
\begin{split}
C_1:\qquad & y^9=x(x-1)^3,\\
C_2:\qquad & y^7=x(x-1)^2,\\
C_3:\qquad &y^7=x(x-1).
\end{split}
\]
 An alternative equation for $C_1$ is
\begin{equation}\label{eq:C1alt}
y^3=z^4-z,\qquad \text{ where }z^3=x.
\end{equation}

We put $K_{N_i}=\Q(\zeta_{N_i})$ and $G_{N_i}=(\Z/N_i\Z)^\ast$. In the
three cases we consider in Lemma \ref{lem:exa}, we have
$G_{N_i}\simeq C_6$. For $j\in (\Z/N_i\Z)^\ast$, we denote the corresponding
element of $\Gal(K_{N_i}/\Q)$ by
\[
\sigma_j :\zeta_{N_i}\mapsto \zeta_{N_i}^j,
\]
or also by $j$ when no confusion can arise.

The following lemma summarizes the properties of the curves
$C_i$. 

\begin{lemma}\label{lem:exa}
\begin{itemize}
\item[(a)] The curve $C_1$ has CM by $\Q(\zeta_9)$. The CM-type is
  $(1,2,4)$. This type is primitive.
\item[(b)] The curve $C_2$ has CM by $\Q(\zeta_7)$ and CM-type
  $(1,2,4)$. This type is imprimitive.
\item[(c)]  The curve $C_3$ has CM by $\Q(\zeta_7)$ and CM-type
  $(1,2,3)$. This type is primitive.
\end{itemize}
\end{lemma}

\begin{proof} 
It is easy to check that the automorphism $\alpha$ of $C_i$ has a fixed
point. Using this point to embed the curve $C_i$ in its Jacobian, we
see that $\alpha$ induces an endomorphism $\alpha\in \End(\Jac(C_i))$
of multiplicative order $N_i$. 

We may regard
 $\alpha\in \End(\Jac(C_i))$ as a primitive $N_i$th root of unity. In
 all three cases, we have
 $2g(C_i)=6=\varphi(N_i)=[\Q(\zeta_{N_i}):\Q]$. It follows that $C_i$
 has CM by $K_{N_i}$.

To calculate the CM-type of $C_i$ we follow the strategy of Section
1.7 of Chapter 1 of \cite{Lang}, and identify the cohomology group
$H^0(C_i, \Omega)$ of holomorphic differentials with the tangent space
of $\Jac(C)$. It suffices to find a basis of $H^0(C_i, \Omega)$
consisting of eigenvectors of $\alpha^\ast$, the map induced by
$\alpha$ on $H^0(C_i,\Omega)$.  Such a basis is computed in Section
1.7 of Chapter 1 of \cite{Lang}. The statement on the CM-type easily
follows from this. (The fact that the action of $\langle
\alpha\rangle$ on $H^0(C_i, \Omega)$ does not factor through the
action of a quotient group provides a second proof that $\alpha$
defines an endomorphism of order $N_i$ of $\Jac(C_i)$.)

We explain what happens for $C_1$. We use a slightly different
notation from Theorem 1.7.1 of Chapter 1 of  \cite{Lang}. A basis of $H^0(C_1,
\Omega)$ is given by
\[
\omega_1=\frac{y\, {\rm d} x}{x(x-1)},\qquad \omega_2=\frac{y^2\, {\rm
    d} x}{x(x-1)}, \qquad \omega_4=\frac{y^4\, {\rm d} x}{x(x-1)^2}.
\]
Note that $\alpha^\ast\omega_i=\zeta_9^i\omega_i$.  The statement on
the CM-type of $\Jac(C_1)$ follows. Primitivity is shown in Section~\ref{subsec:Primitive types in Case 1}.

In Example \ref{eg:zeta7} we have determined all primitive CM-types for $\Q(\zeta_7)$.
The statements on the (im)primitivity of the CM-types of $C_2$ and $C_3$ follow
from this.
\end{proof}

\begin{remark} Lemma \ref{lem:exa}.(b) implies that $\Jac(C_2)$ is not simple.
 We
 may also  check this directly. The curve $C_2$ admits an automorphism
 \[
\beta(x,y)=\left( \frac{1}{1-x}, \frac{y^2}{1-x}\right).
\]
The curve $E:=C_2/\langle \beta\rangle$ has genus $1$. This curve has
CM by the field $K_1=\Q(\zeta_7)^{\langle \sigma_2\rangle}=\Q(\sqrt{-7})$.

One checks that $\langle \alpha, \beta\rangle\simeq \mathbb{Z}/7\mathbb{Z} \rtimes \mathbb{Z}/3\mathbb{Z}$ is
a non-abelian group. Using the method of Kani--Rosen (\cite{KaniRosen} or
\cite{Paulhus}), one may also deduce from this that
\[
\Jac(C_2)\sim E^3.
\]
\end{remark}

Our next goal is to describe the reduction behavior of
the curves $C_1$ and $C_3$.

\begin{lemma}\label{lem: C1}
\begin{itemize}
\item[(a)] The curve $C_1$ has bad reduction at $p=3$ and good
  reduction at all other primes.
\item[(b)] The reduction $\oJ_{1,p}$ of the Jacobian $J_1$ of
  $C_1$ to characteristic $p$ is ordinary if and only if $p\equiv
   1\pmod{9}$ and supersingular if and only if $p=3$ or $p\equiv
  2\pmod{3}$.
\item[(c)] If $p\equiv 4,7\pmod{9}$, then the abelian variety
  $\oJ_{1,p}$ is simple.
\end{itemize}
\end{lemma}

\begin{proof}
It is easy to see that $C_1$ has good reduction to characteristic
$p\neq 3$. Indeed, (\ref{eq:C1alt}) still defines a smooth projective
curve in characteristic $p\neq 3$. We consider the reduction at
$p=3$. In this case, the extension of function fields
\[
\F_3(z)\subset \F_3(z)[y]/(y^3-z(z^3-1))
\]
defines a purely inseparable field extension. This implies that
$\F_3(z)[y]/(y^3-z(z^3-1))$ is the function field of a curve of genus
$0$. This does not imply that $C_1$ has bad
reduction to characteristic $3$, since there could be a different
model.

We claim that there does not exist a curve of genus $3$ in
characteristic $3$ with an automorphism of order $9$. This claim
implies that $C$ has bad reduction to characteristic $3$. Indeed, if
$C$ has potentially good reduction, then the automorphism group
$\Aut(\oC)$ of the reduction $\oC$ of $C$ contains $\Aut(C)$. Hence,
in particular, $\Aut(\oC)$ contains an automorphism of order $9$.

To obtain a contradiction, we assume that $X$ is a curve of genus $3$
in characteristic $3$ with an automorphism $\gamma$ of order $9$. We
consider the Galois cover
\[
X\to X/\langle\gamma\rangle.
\]
This cover is wildly ramified of order $9$ above at least one point.  We
apply the Riemann--Hurwitz formula to this cover.  It follows from
Theorem 1.1 of \cite{ObusPries} that the contribution of a wild
ramification point with ramification index $9$ to $2g(X)-2$ in the
Riemann--Hurwitz formula is at least $2\cdot (9-1)+ 5\cdot (3-1)=26$,
which contradicts the assumption that $X$ has genus $3$. This proves
(a).

We have shown that $C_1$ has bad reduction to characteristic $3$. Let
$\oC_{1, 3}$ be the stable reduction of $C_1$ to characteristic
$3$. Then $\oC_{1, 3}$ contains at least $2$ irreducible components of
positive genus (Corollary \ref{cor:SerreTate}). Furthermore, there is
an automorphism of order $9$ acting on $\oC_{1,3}$. The only way this
is possible is if $\oC_{1,3}$ contains three irreducible components of
positive genus, which are then elliptic curves, each with an automorphism
of order $3$.  The automorphism of order $9$ permutes these
components.  There is a unique elliptic curve with an automorphism of
order $3$, namely the elliptic curve with $j=0$. In characteristic $3$
this curve may be given by
\begin{equation}\label{eq:sschar3}
w^3-w=v^2.
\end{equation}
This curve is supersingular by the Deuring--Shafarevich formula
(\cite{Crew}).  We conclude that the reduction $\oJ_{1,3}$ of the
Jacobian of $C_1$ to characteristic $3$ is supersingular. Proposition
\ref{prop:SerreTate}.(b) implies that $\oJ_{1,3}$ is in fact
superspecial: the Jacobian $\oJ_{1,3}$ is isomorphic to three copies
of the supersingular elliptic curve (\ref{eq:sschar3}) as a polarized
abelian variety.

The rest of (b) may be deduced from \cite{Yui}. For $p\equiv
4,7\pmod{9}$, Yui's results \cite{Yui} imply that $\oJ$ is neither
ordinary nor supersingular. In fact, her results imply that $\oJ$ has
$p$-rank zero, but is not supersingular. Theorem \ref{thm: Lang2}
therefore implies that $\oJ$ is simple.
\end{proof}

The situation for $C_3$ is similar but somewhat easier.

\begin{lemma}\label{lem: C3}
\begin{itemize}
\item[(a)] The curve $C_3$ has good reduction at $p\neq 7$ and
  potentially good reduction at $p=7$.
\item[(b)] The reduction $\oJ_{3,p}$ of the Jacobian $J_3$ of
  $C_3$ to characteristic $p$ is ordinary if and only if $
  p\equiv 1\pmod{7}$ and supersingular if and only if $p=7$ or $p\equiv
  -1,3,5\pmod{7}$.
\end{itemize}
\end{lemma}

\begin{proof}
The fact that $C_3$ has good reduction to characteristic $p\neq 7$
follows as in the proof of Lemma \ref{lem: C1}. The curve $C_3$ has
potentially good reduction to characteristic $7$ as well, see Example
3.8 of \cite{Arizona}. The curve $C_3$ does not have good reduction
over $\Q_7$ but acquires good reduction over the extension
$\Q_7(\zeta_7)$ of $\Q_7$.

Statement (b) for $p\neq 7$ follows from \cite{Yui}.
We consider the reduction $\oC_{3, 7}$ of $C$ to characteristic $7$.
 In characteristic
$7$, the reduction $\oC_{3, 7}$ is given by
\[
w^7-w=v^2,
\]
by Example 3.8 of \cite{Arizona}. By the Deuring--Shafarevich formula,
it follows that the Jacobian $\oJ_{3,7}$ of $\oC_{3,7}$ has $p$-rank
$0$. To show that it is supersingular, it suffices to find an elliptic
quotient of the curve $\oC_{3,7}$.

The curve $\oC_{3,7}$ admits an extra automorphism of
order $3$ given by
\[
\beta(v,w)=(\zeta_3 v, \zeta_3^2 w),
\]
where $\zeta_3\in \F_7^\times$ is an element of order three. The
automorphism $\beta$ has exactly two fixed points, namely the points
with $w=0, \infty$. It follows that
$E_{3,7}:=\oC_{3,7}/\langle\beta\rangle$ is an elliptic curve. This
shows that $\oJ_{3,7}$ is supersingular.
\end{proof}

\subsection{A Picard curve example}\label{sec:Picardexa}
We end this section by considering Example 3 from Section 5 of
\cite{KoikeWeng}, wherein Koike and Weng study Picard curves with CM. We  show
that the curve in the aforementioned example has bad reduction to
 characteristic $p=5$, and that
the stable reduction consists of an elliptic curve and a curve of genus
$2$. We will show that the Jacobian has superspecial reduction in this
case. This is an example where the reduction $\oJ$ of the Jacobian is
isomorphic to $E^3$, but the polarization is neither that of a smooth
curve nor the product polarization $E\times \{0\}\times \{0\}+
\{0\}\times E\times \{0\}+\{0\}\times \{0\}\times E$.

 A Picard
curve is a curve of genus $3$ given by an equation
\[
y^3=f(x),
\]
where $f(x)\in \C[x]$ is a polynomial of degree $4$ with simple
roots. Every Picard curve admits an automorphism $\alpha(x,y)=(x,
\zeta_3y)$. Therefore, the endomorphism ring of the Jacobian contains
$\Q(\zeta_3)$.

Let $C_4$ be the smooth projective curve defined by
\[
y^3=f(x):=x^4-13\cdot 2\cdot 7^2\cdot x^2+2^3\cdot 13\cdot 5\cdot 47\cdot
x-5^2\cdot 31\cdot 13^2.
\]
Koike and Weng show that the Jacobian of $C_4$ has CM by the field
$K=K^+K_1$ with $K_1=\Q(\zeta_3)$ and $K^+=\Q[t]/(t^3-t^2-4t-1)$.  The
CM-field $K$ is Galois over $\Q$, hence we are in Case 1 of
Proposition \ref{prop:classification}. One may show that the
corresponding CM-type is primitive. For example, one may check using
\cite{Bouwthesis} that the reduction $\oJ_{4,7}$ of the Jacobian $J_4$
of $C_4$ to characteristic $7$ has $p$-rank $1$, and hence is neither
ordinary nor supersingular. It follows from this that the Jacobian
$J_4$ is simple. The primitivity of the CM-type follows from this, by Theorem \ref{thm: primsimple}.

We now consider the reduction of $C_4$.  The discriminant of $f$ is
$2^{12}\cdot 5^6\cdot 13^4$ which shows that $C_4$ has good reduction
for $p\neq 2,3,5,13$.  One may check that $C_4$ also has good
reduction at $p=2,13$. We do not consider what happens for $p=3$.

 We determine the reduction at $p=5$. Note that
\begin{equation}\label{eq: g=2}
f(x)\equiv x^2(x+2)(x-2)=x^4+x^2\pmod{5}.
\end{equation}
Therefore, the stable reduction of  $C_4$ contains an irreducible
component $\overline{D}$ of genus $2$ given by the equation
\begin{equation}\label{oDeq}
\bar{y}^3=\bar{x}^2(\bar{x}^2+1).
\end{equation}
The reason that this curve has genus $2$ rather than $3$ is that the
$3$-cyclic cover $(\bar{x},\bar{y})\mapsto \bar{x}$ has only $4$
branch points in characteristic $5$, and not $5$ branch points as it had in
characteristic zero. It follows that the curve $C_4$ has bad reduction to characteristic
$5$, and the reduction of $C_4$ consists of the curve $\overline{D}$
of genus $2$ intersecting with an elliptic curve. (We do not actually
have to compute the elliptic component to conclude this.)  The reduction
$\oJ_4$ of the Jacobian of $C_4$ is therefore isogenous to the product
of an elliptic curve and the abelian surface $\Jac(\overline{D})$.  To
determine the reduction type of $\oJ_4$, we first consider the
Jacobian $\Jac(\overline{D})$ of the curve $\overline{D}$ given by the
equation (\ref{oDeq}).

 One may show by computing the Hasse--Witt matrix of $\overline{D}$
 that the Jacobian $J(\overline{D})$ is supersingular. This is a
 similar calculation to the one we did in Section
 \ref{sec:cyclicexa}. However, since $\overline{D}$ has genus $2$, it
 suffices to compute the $p$-rank. In fact, the Hasse--Witt matrix is
 identically zero, which shows that $J(\overline{D})$ is superspecial,
 i.e., isomorphic to the product of two supersingular elliptic
 curves.

Alternatively, we may note that $\overline{D}$ has  additional
automorphisms given by
\[
 \tau(\bar{x}, \bar{y})=(-\bar{x}, \bar{y}), \qquad \rho(\bar{x},
 \bar{y})=\left(-\frac{1}{\bar{x}}, \frac{\bar{y}}{\bar{x}^2}\right), \qquad
\tau\circ\rho(\bar{x},
 \bar{y})=\left(\frac{1}{\bar{x}}, \frac{\bar{y}}{\bar{x}^2}\right).
\]
Note that $\tau$ fixes the two points with $\bar{x}=0, \infty$ and
$\rho$ fixes the two points with $\bar{x}^2=-1$.
The quotients
$\oC_4/\langle \tau\rangle$ and $\oC_4/\langle \rho\rangle$ are
elliptic curves, each with an automorphism of order $3$. In
particular, these elliptic curves have $j=0$. Since $p=5\equiv
2\pmod{3}$, they are supersingular.  Theorem \ref{thm: Lang2} implies
that $\oJ$ is isogenous to $E_0^3$, where $E_0$ denotes the
supersingular elliptic curve over $\overline{\F_5}$ with $j=0$.

\begin{remark}
\label{remark:CM_subfield}
The examples we discussed in this section all have the property that
the CM-field $K$ contains a CM-subfield $K_1$ with $\Q\subsetneq
K_1\subsetneq K$. In Section \ref{subsec:degenerate}, we will show that this implies that the embedding
problem, which we formulate in Section \ref{sec:embedding}, has
degenerate solutions for every prime. This explains why we exclude this case in Theorem \ref{thm:bnd_on_p}.
\end{remark}

\begin{remark} 
\label{remark:CMconstruction}
Let $K$ be a sextic CM-field. It is known how to construct genus $3$ curves $C$ in
characteristic $0$ with CM by $K$ (\cite{Shimura}, Sections 6.2 and
14.3). We sketch the construction. 

We fix a CM-type $(K,\varphi)$. Let $\delta_{K/\Q}$ be the different.  For any ideal
${\mathfrak a}$ of ${\mathcal O}_K$ we consider the lattice
$\varphi({\mathfrak a})=(\varphi_1({\mathfrak a}),
\varphi_2({\mathfrak a}),\varphi_3({\mathfrak a}))$. Then
\[
A:=\C^3/\varphi({\mathfrak a})
\]
is an abelian variety with CM by $(K, \varphi)$.  Shimura (Theorem 3
of Section 6.2 in \cite{Shimura}) shows that all CM abelian varieties
occur in this way.

In Section 14.3 of \cite{Shimura}, Shimura also describes all Riemann
forms defining principal polarizations on $A$. Such a Riemann form
exists if the following two conditions are satisfied.
\begin{itemize}
\item The ideal $\delta_{K/\Q}{\mathfrak a}\overline{{\mathfrak
    a}}=(a)$ is principal.
\item There exists a unit $u\in {\mathcal O}_K$ such that $ua$ is totally
  imaginary and the imaginary part of $\varphi_i(ua)$ is negative for
  all $i$.
\end{itemize}

Every principally polarized
abelian variety of dimension $3$ is isomorphic to the Jacobian of a
(possibly singular) genus 3 curve $C$ by Theorem 4 of
\cite{OortUeno}.
More precisely, Oort and Ueno show that the curve $C$ is of compact
type, meaning that $A$ is isomorphic to the product of the Jacobians of
the irreducible components of positive genus of $C$. (This notion is
essentially the same as the notion ``tree-like'' that we used in Section
\ref{sec:curvered}.)  In our situation, the abelian variety $A$ is
simple, and it follows that the curve $C$ is smooth.
\end{remark}

\section{Embedding problem} 
\label{sec:embedding}

\subsection{Formulation of the embedding problem}\label{61}

Let $C$ be a genus $3$ curve defined over some number field $M$. We
assume that the Jacobian $J=\Jac(C)$ has CM by a sextic CM-field
$K$. After replacing $M$ by a finite extension if necessary, we may
assume that $J$ has good reduction (Theorem \ref{thm:SerreTate}) and
that $C$ has stable reduction at all finite places of $M$.

In this section, we make the following important assumption.

\begin{assumption}\label{ass:embedding1}
 We assume that $K$ does not contain an imaginary quadratic
subfield.
\end{assumption}

Recall that Assumption \ref{ass:embedding1} implies that the CM-type
of $C$ is primitive (Corollary \ref{cor:K_1 in K}). The reason for
making this assumption is discussed in Section \ref{subsec:degenerate}.

Let $\mathfrak{p}$ be a finite prime of $M$ where the curve $C$ has
bad reduction. We write $\overline{k}$ for the algebraic closure of
the residue field at $\frp$ and let $p$ denote the residue characteristic. We want
to bound these primes $p$. (See Theorem \ref{thm:bnd_on_p} for the
precise statement of our result.) Recall from Corollary
\ref{cor:SerreTate} that there are two possibilities for the reduction
$\oC$ of $C$. In this section, we only deal with the case where $\oC$
has three irreducible components of genus $1$ and
postpone the other case for future work. To summarize, we make the following assumption on the prime
$\mathfrak{p}$.

\begin{assumption}\label{ass:embedding2}
 Let $\mathfrak{p}$ be a finite prime of $M$, such that the stable
 reduction $\oC=\oC_{\mathfrak{p}}$ of $C$ at $\mathfrak{p}$ contains
 three elliptic curves as irreducible components (Case (ii) of Corollary
\ref{cor:SerreTate}).
\end{assumption}

Let $\mathfrak{p}$ be as in Assumption \ref{ass:embedding2}.  We write
$E_1, E_2, E_3$ for the three elliptic curves that are the irreducible
components of $\oC$. We write $\oJ$ for the reduction of $J$ at
$\mathfrak{p}$. Recall from Remark \ref{rem:SerreTate} that we have an
isomorphism
\[
\oJ\simeq E_1\times E_2\times E_3
\]
as polarized abelian varieties, i.e., the polarization on $\oJ$ is the
product polarization. Corollary \ref{cor:K_1 in K} implies that the
$E_i$ are supersingular. In particular, they are isogenous. (This also
follows from Theorem \ref{thm: Lang2}).

Let $\End(J)=\mathcal{O}\subset\mathcal{O}_K$. Reduction at the prime
$\frp$ gives an injective ring homomorphism
\[
\mathcal{O}\hookrightarrow\End(\overline{J})\simeq\End(E_1\times E_2\times E_3).
\]

\begin{problem}[The embedding problem]\label{problem:embedding}
Let ${\mathcal O}$ be an order in a sextic CM-field $K$, and let $p$
be a prime number.  The {\em embedding problem} for $\cO$ and $p$ is
the problem of finding elliptic curves $E_1, E_2, E_3$
defined over a field of characteristic $p$, and a ring embedding
\[
i: {\mathcal O}\hookrightarrow \End(E_1\times E_2\times E_3)
\]
such that the Rosati involution on $\End(E_1\times E_2\times E_3)$ induces complex conjugation on ${\mathcal O}$. We call such a ring
embedding a {\em solution to the embedding problem} for $\cO$
and $p$.
\end{problem}

The following result states that if we have a solution to the
embedding problem then the elliptic curves $E_i$ are automatically
isogenous. The proof we give here works directly with the abelian
variety $E_1\times E_2\times E_3$ without considering it as the
reduction of an abelian variety in characteristic zero. However the
proof is essentially the same as the proofs of Theorem \ref{thm: Lang2}
and Proposition \ref{prop: K_1 in K}.

\begin{lemma}
\label{lem:supersingular}
Let $K$ be a sextic CM-field. Suppose that there exist elliptic curves $E_1, E_2, E_3$
defined over a field of characteristic $p>0$ and an injective $\Q$-algebra homomorphism
\[
i: K \hookrightarrow \End^0(E_1\times E_2\times E_3).
\]
Then the elliptic curves $E_1$, $E_2$ and $E_3$ are all isogenous. Furthermore, if $K$ contains no imaginary quadratic subfield then the $E_i$ are supersingular.
\end{lemma}

\begin{proof}
First suppose that no two of the elliptic curves $E_1$, $E_2$, $E_3$ are isogenous. Then
\[i: K\hookrightarrow \End^0(E_1\times E_2\times E_3)=\begin{pmatrix}\End^0 E_1 & 0 & 0\\ 0 & \End^0 E_2 & 0\\ 0& 0 &\End^0 E_3\end{pmatrix}=\End^0 E_1\times \End^0 E_2\times \End^0 E_3. \]

Projecting on the factor $\End^0 E_i$ gives a ring homomorphism $K\hookrightarrow \End^0 E_i$. Since $K$ is a field, this ring homomorphism must be injective. 
But $\End^0 E_i$ is either an imaginary quadratic field or a quaternion algebra, neither of which can contain a sextic field.

Now suppose that exactly two of the elliptic curves are isogenous. Without loss of generality, we may assume that $E_1\sim E_2$ and $E_1\not\sim E_3$. Then
\[i: K\hookrightarrow \End^0(E_1\times E_2\times E_3)=\begin{pmatrix}\End^0 E_1 & \End^0 E_1 & 0\\ \End^0 E_1 & \End^0 E_1 & 0\\ 0& 0 &\End^0 E_3\end{pmatrix}=M_2(\End^0 E_1)\times \End^0 E_3. \]
Again, projecting on the factor $\End^0 E_3$, we see that $K\hookrightarrow \End^0 E_3$. This is impossible for dimension reasons. Thus, we have proved that all three elliptic curves are isogenous.

Now suppose that $K$ contains no imaginary quadratic subfield and that
the elliptic curves $E_i$ are ordinary. Then $\End^0 E_1=K_1$ for some
imaginary quadratic field $K_1$ and
\[i: K\hookrightarrow \End^0(E_1\times E_2\times E_3)=M_3(K_1).\]
Let $\beta$ be a generator for $K$ over $\Q$ and let $f$ be its
minimal polynomial, which has degree $6$. The matrix $i(\beta)\in
M_3(K_1)$ has a minimal polynomial of degree at most $3$ over
$K_1$. Since $i$ is an injective $\Q$-algebra homomorphism, this means
that $f$ splits over $K_1$. Since $K_1$ is quadratic, this implies
that $K_1\hookrightarrow K$, contradicting
the assumption that $K$ contains no imaginary quadratic subfield.
\end{proof}

\begin{proposition}\label{prop:assumptions}
Let $C$ be a genus $3$ curve such that ${\mathcal O}:=\End(\Jac(C))$
is an order in a sextic CM-field $K$ satisfying Assumption
\ref{ass:embedding1}.  Let $M$ be a number field over which $C$ is
defined, and let $\mathfrak{p}$ be a prime of bad reduction of $C$
such that Assumption \ref{ass:embedding2} is satisfied. Write $p$ for
the residue characteristic of $\mathfrak{p}$. Then there exists a
solution to the embedding problem for $\cO$ and $p$. Moreover, in this situation the
three elliptic curves are supersingular.
\end{proposition}

\begin{proof}  Let $C$ be as in the statement of the proposition. Then the
CM-type of its Jacobian $J$ is primitive (Corollary \ref{cor:K_1 in
  K}.(a)). Therefore the Rosati involution acts as complex conjugation
on $\End^0(J)=K$ by Proposition \ref{prop: Rosati}. The canonical
polarization on the Jacobian $J$ is a principal polarization,
therefore the Rosati involution also acts on $\End(J)={\mathcal O}$.

Assumption \ref{ass:embedding2} implies that the reduction
$\overline{J}$ of the Jacobian at $\mathfrak{p}$ is isomorphic to a
product of three elliptic curves $E_i$ as polarized abelian
varieties. These elliptic curves are supersingular (Corollary
\ref{cor:K_1 in K}.(b)).  Remark \ref{rem: Rosati} shows that we
obtain a solution to the embedding problem.
\end{proof}

\subsection{Endomorphisms of $\overline{J}$ as $3\times 3$ matrices}
\label{sec:endo}
 In this section we describe the ring $\End(E_1\times E_2\times E_3)$
 from the embedding problem (Problem \ref{problem:embedding}). Recall
 that we may assume that the $E_i$ are isogenous (Lemma
 \ref{lem:supersingular}). We recall from Proposition \ref{prop: Rosati2} the
 description of the Rosati involution corresponding to the product
 polarization on $E_1\times E_2\times E_3$.

We can view an element $f \in \End(E_1 \times E_2 \times E_3)$ as a
matrix
$$
f=\begin{pmatrix}
f_{1,1} & f_{1,2} & f_{1,3} \\
f_{2,1} & f_{2,2} & f_{2,3} \\
f_{3,1} & f_{3,2} & f_{3,3}
\end{pmatrix}, $$
where $f_{i,j} \in \Hom(E_j,E_i)$. Given two endomorphisms $f,g$ the
composition $ f \circ g$ corresponds to multiplication of
matrices. Since the polarization on $\oJ=E_1\times E_2\times E_3$ is
the product polarization, the Rosati involution $f \mapsto f^{\ast}$
sends $f$ to
$$\begin{pmatrix}
f_{1,1}^{\vee} & f_{2,1}^{\vee} & f_{3,1}^{\vee} \\
f_{1,2} ^{\vee}& f_{2,2} ^{\vee}& f_{3,2}^{\vee} \\
f_{1,3} ^{\vee}& f_{2,3}^{\vee} & f_{3,3}^{\vee}
\end{pmatrix} $$
where $f_{i,j}^{\vee}$ denotes the dual isogeny of $f_{i,j}$.

For $i=2,3$, let $\psi_i:E_1\rightarrow E_i$ be an isogeny of degree
$\delta_i$. Let $f\in \End(E_1\times E_2\times E_3)$. Then the
composition
\[
\xymatrix{
E_1\times E_1\times E_1\ar[r]^{(1,\psi_2,\psi_3)} & E_1 \times
E_2\times E_3
\ar[rr]^{(1,\delta_2^{-1}\psi_2^{\vee},\delta_3^{-1}\psi_3^{\vee})} &&
E_1\times E_1\times E_1
}
\]
induces an injective $\mathbb{Q}$-algebra homomorphism
\begin{equation}
\label{eq: Ei to E1}
\End^0 (E_1\times E_2\times E_3)\hookrightarrow \End^0(E_1\times E_1\times E_1)=M_3(\End^0 E_1).
\end{equation}
Let $\Phi$ denote the composite map
\[\Phi: K\hookrightarrow\End^0(E_1\times E_2\times E_3)\hookrightarrow M_3(\End^0 E_1).\]
It is easily seen that
$$\begin{pmatrix}
1& 0 & 0 \\
0& \delta_2 & 0 \\
0& 0 & \delta_3
\end{pmatrix}\Phi(\calO)\subset M_3(\End E_1).$$

Under the assumptions made in Section \ref{61}, we may assume that the
elliptic curves $E_i$ in the formulation of the embedding problem are
supersingular (Proposition \ref{prop:assumptions}). We therefore
recall some well-known facts on the endomorphism ring of a
supersingular elliptic curve.

 Let $p\in\mathbb{Z}_{>0}$ be the rational prime
lying below $\frp$.

\begin{proposition}
\label{prop:quaternion algebra}
Let $E$ be a supersingular elliptic curve defined over a field of
characteristic $p$. 
 Then $\End^0 E$ is a quaternion algebra over $\Q$
ramified at precisely the places $\{p, \infty\}$. This quaternion
algebra is non-canonically isomorphic to the algebra $B_{p,\infty}$,
where $B_{p,\infty}=(\frac{-1,-1}{\mathbb{Q}})$ if $p=2$ and if $p$ is
odd, $B_{p,\infty}=(\frac{-\varepsilon,-p}{\mathbb{Q}})$ where
\[
\varepsilon=\begin{cases} 1 & \textrm{if}\ p\equiv 3\pmod{4},\\ 2 &
\textrm{if}\ p\equiv 5\pmod{8},\\ \ell &
\textrm{if}\ p\equiv 1\pmod{8}.
\end{cases}
\]
In the case that $p\equiv 1\pmod{8}$, $\ell\in\mathbb{Z}_{>0}$ is a
prime such that $\ell\equiv 3 \pmod{4}$ and $\ell$ is not a square
modulo $p$.  Any isomorphism sends $\End E$ to an order of
$B_{p,\infty}$ and the involution given by taking the dual isogeny
corresponds to the canonical involution on $B_{p,\infty}$.
\end{proposition}

\begin{proof}
The fact that the endomorphism algebra $\End^0(E)$ of a supersingular
elliptic curve is a quaternion algebra over $\Q$ ramified precisely at
$\{p,\infty\}$ is proved, for example, in Section 21 of \cite{MAV}. The
statement on the Rosati involution is also proved in loc. cit. The
uniqueness of the quaternion algebra is proved, for example, in Theorem
III.3.1 of \cite{Vigneras}.

For $p=2$, let $Q=(\frac{-1,-1}{\mathbb{Q}})$. For every odd prime $p$, let $\varepsilon$ be as in the statement of the
proposition and let $Q=(\frac{-\varepsilon,-p}{\mathbb{Q}})$ be the
corresponding quaternion algebra. The statement that $Q$ is exactly
ramified at the places $\{p,\infty\}$ follows easily from the
properties of the Hilbert symbol (page 37 of
\cite{Vigneras}).
\end{proof}

For $b\in B_{p,\infty}$, we write $\Nrd(b) = bb^*$ (where $b^*$
represents the involution on the quaternion algebra) for the reduced
norm of $b$.  The reduced norm corresponds to the degree of an
endomorphism under the identification in
Proposition~\ref{prop:quaternion algebra}.

\begin{lemma}[Elements of small norm commute] \cite[Corollary 2.1.2]{GorenLauter07}
\label{smallnorms}
Let $R$ be a maximal order of $B_{p,\infty}$. If $k_1,k_2 \in R$ and $\Nrd(k_1), \Nrd(k_2) < \sqrt{p}/2$
then $k_1k_2 = k_2k_1$.
\end{lemma}

\subsection{Bounding the primes of bad reduction for $C$}\label{62}
Recall that $J=\Jac(C)$ is the Jacobian of a genus $3$ curve $C$ which
has complex multiplication by an order $\calO$ in a sextic CM-field $K$ which does not
contain an imaginary quadratic field (Assumption \ref{ass:embedding1}).
Let $K^+$ denote the totally real cubic subfield of $K$. The main result of this
section is Theorem \ref{thm:bnd_on_p} which gives an upper bound on
the primes of bad reduction for $C$ satisfying Assumption
\ref{ass:embedding2}.

\begin{theorem}
\label{thm:bnd_on_p}
Suppose that $K$ does not contain an imaginary quadratic subfield. Let $\mathfrak{p}\mid p$ be a prime of bad reduction for $C$ satisfying Assumption \ref{ass:embedding2}. Write
$K=\mathbb{Q}(\sqrt{\alpha})$ for some totally negative element
$\alpha\in K^+\setminus\mathbb{Z}$ with $\sqrt{\alpha}\in\cO=\End (J)$. Then
$p\leq 4\Tr_{K^+/\Q}(\alpha)^6/3^6.$
\end{theorem}

The existence of such $\alpha$ is guaranteed because the sextic CM-field $K$ contains no imaginary quadratic subfield. By Proposition \ref{prop:assumptions}, the following result implies Theorem \ref{thm:bnd_on_p}.

\begin{theorem}
\label{thm:no embedding}
Suppose that $K$ does not contain an imaginary quadratic
subfield. Let $p$ be a prime such that there exists a solution to the embedding
problem (Problem \ref{problem:embedding}) for some order $\cO$ of
$K$. Write $K=\mathbb{Q}(\sqrt{\alpha})$ for some totally negative
element $\alpha\in K^+\setminus\mathbb{Z}$ with
$\sqrt{\alpha}\in\cO$. Then $p\leq 4\Tr_{K^+/\Q}(\alpha)^6/3^6.$
\end{theorem}

We break down the proof of Theorem \ref{thm:no embedding} into several lemmas.

Let $$Q=\begin{pmatrix}\
 r & s & t  \\
 u & v & w  \\
 x & y & z
\end{pmatrix}
$$ be the image of $\sqrt{\alpha}$ in $\End(E_1\times E_2\times E_3)$. By Proposition \ref{prop: Rosati}, the Rosati involution corresponds to complex conjugation on $K$, so we have
\begin{equation}
\label{Rosati}
\begin{pmatrix}\
 r^{\vee} & u^{\vee} & x^{\vee} \\
 s^{\vee} & v^{\vee} & y^{\vee}  \\
 t^{\vee} & w^{\vee} & z^{\vee}
\end{pmatrix}=\begin{pmatrix}\
 -r & -s &- t  \\
 -u & -v & -w  \\
 -x & -y & -z
\end{pmatrix} .
\end{equation}

\begin{lemma}
\label{lem: st nonzero}
We may assume that the homomorphisms $s: E_2\rightarrow E_1$ and $t:E_3\rightarrow E_1$ are both nonzero.
\end{lemma}

\begin{proof}
Suppose for contradiction that both $s$ and $t$ are zero. Then the image of $\alpha$ in $\End(E_1\times E_2\times E_3)$ is
$$Q^2=\begin{pmatrix}\
 -rr^{\vee} & 0 & 0  \\
 0 & -vv^{\vee}-ww^{\vee} & vw+wz  \\
 0 & -w^{\vee}v-zw^{\vee} & -w^{\vee}w-zz^{\vee}
\end{pmatrix}.$$
For $i=2,3$, let $\psi_i:E_1\rightarrow E_i$ be an isogeny of degree $\delta_i$. As seen in \eqref{eq: Ei to E1}, the $\psi_i$ induce an injective $\mathbb{Q}$-algebra homomorphism $\End^0(E_1\times E_2\times E_3)\rightarrow \End^0(E_1\times E_1\times E_1)=M_3(\End^0 E_1)$ sending
$Q^2$ to
$$S=\begin{pmatrix}\
 -rr^{\vee} & 0 & 0  \\
 0 & -vv^{\vee}-ww^{\vee} & \delta_2^{-1}\psi_2^{\vee}(vw+wz)\psi_3  \\
 0 & \delta_3^{-1}\psi_3^{\vee}(-w^{\vee}v-zw^{\vee})\psi_2 & -w^{\vee}w-zz^{\vee}
\end{pmatrix} .$$
Since $(vw+wz)^{\vee}=-w^{\vee}v-zw^{\vee}$, the entries of $S$ commute and therefore form a subfield $L$ of $\End^0 E_1$.
Since $S$ is the image of $\alpha$ under an injective $\mathbb{Q}$-algebra homomorphism, the minimal polynomial of $S$ over $L$ divides the minimal polynomial of $\alpha$ over $\mathbb{Q}$. Recall that $rr^{\vee}\in\mathbb{Z}$ is the degree of $r$. Now $-rr^{\vee}$ is an eigenvalue of $S$ and therefore a root of its minimal polynomial. But this means that the minimal polynomial of $\alpha$ over $\mathbb{Q}$ has a root in $\mathbb{Z}$, contradicting its irreducibility.

Therefore, at least one of $s,t$ is nonzero. Using $E_2$ in place of $E_1$, we see that at least one of $s,w$ is nonzero. Using $E_3$ in place of $E_1$, we see that at least one of $t,w$ is nonzero. Putting all these conditions together and reordering the elliptic curves $E_1, E_2, E_3$ if necessary, we may assume that $s$ and $t$ are both nonzero.
\end{proof}

Henceforth, we assume that $s$ and $t$ are nonzero.
Therefore, we can use $s^{\vee}$ and $t^{\vee}$ to give an injective $\mathbb{Q}$-homomorphism $\End^0(E_1\times E_2\times E_3)\hookrightarrow \End^0(E_1\times E_1\times E_1)$ as in \eqref{eq: Ei to E1}. The image of $\sqrt{\alpha}$ in $M_3(\End^0 E_1)$ is

\begin{equation}
\label{eq: matrix T}
T=\begin{pmatrix}\
 r & \delta_2 & \delta_3  \\
 -1 & svs^{\vee}/\delta_2 & swt^{\vee}/\delta_2  \\
 -1 & -tw^{\vee}s^{\vee}/\delta_3 & tzt^{\vee}/\delta_3
\end{pmatrix},
\end{equation}
where $\delta_2=\deg(s)$ and $\delta_3=\deg(t)$. 

Since $K$ contains no imaginary quadratic subfield,
Lemma \ref{lem:supersingular} shows that the elliptic curves $E_1,
E_2$ and $E_3$ are supersingular. By Proposition \ref{prop:quaternion
  algebra}, we may choose an isomorphism $\End^0 E_1\rightarrow
B_{p,\infty}$. The isomorphism sends $\End E_1$ to a maximal order of
$B_{p,\infty}$ and the Rosati involution on $\End E_1$ corresponds to the usual
involution on $B_{p,\infty}$.  We abuse notation slightly by
continuing to write $T$ for the image of $\sqrt{\alpha}$ in
$M_3(B_{p,\infty})$.

\begin{lemma}
\label{lem: noncommutative}
Suppose that $K$ contains no imaginary quadratic subfield.
Let $T$ denote the image of $\sqrt{\alpha}$ in $M_3(B_{p,\infty})$. Then the entries of the matrix $T$ do not all commute with each other.
\end{lemma}

\begin{proof}
Suppose for contradiction that the entries of $T$ commute. Let $K_1$ denote the subfield of $B_{p,\infty}$ generated by the entries of $T$. A subfield of $B_{p,\infty}$ is either $\mathbb{Q}$ or a quadratic subfield which splits $B_{p,\infty}$. But $B_{p,\infty}$ is ramified at the infinite place, so it is not split by any real field. Thus, $K_1$ is either $\mathbb{Q}$ or an imaginary quadratic field. By assumption, $K$ contains no imaginary quadratic subfield. Thus, the minimal polynomial of $\sqrt{\alpha}$ over $\mathbb{Q}$ remains irreducible over $K_1$.

Let $g$ denote the minimal polynomial of $T$ over $K_1$. The degree of $g$ is at most $3$. Since $T$ is the image of $\sqrt{\alpha}$ under an injective $\mathbb{Q}$-algebra homomorphism, $g$ divides the minimal polynomial of $\sqrt{\alpha}$ over $\mathbb{Q}$, which has degree $6$. Thus, the minimal polynomial of $\sqrt{\alpha}$ over $\mathbb{Q}$ factorizes over $K_1$, giving the required contradiction.
\end{proof}

We restrict to the case where $p$ is odd; the case $p=2$ is very similar. By Proposition \ref{prop:quaternion algebra}, $B_{p,\infty}$ has a $\mathbb{Q}$-basis $1,i,j,k$ where $i^2=-\varepsilon$, $j^2=-p$ , $ij=k$, $ji=-ij$ and $\varepsilon$ is as in Proposition \ref{prop:quaternion algebra}.
We embed $B_{p,\infty}$ into $M_4(\mathbb{Q})$ via
$$
1\mapsto \begin{pmatrix} 1 & 0&0&0\\ 0&1&0&0\\0&0&1&0\\0&0&0&1\end{pmatrix},\,
i\mapsto\begin{pmatrix} 0 & -\varepsilon&0&0\\ 1&0&0&0\\0&0&0&-\varepsilon\\0&0&1&0\end{pmatrix},\,
j\mapsto \begin{pmatrix} 0 & 0&-p&0\\ 0&0&0&p\\1&0&0&0\\0&-1&0&0\end{pmatrix} ,\,
k\mapsto  \begin{pmatrix} 0 & 0&0&-\varepsilon p\\ 0&0&-p&0\\0&\varepsilon&0&0\\1&0&0&0\end{pmatrix}.
$$
This induces an embedding $M_3(B_{p,\infty})\hookrightarrow M_{12}(\mathbb{Q})$. Let $U$ denote the image of $\alpha$ in $M_{12}(\mathbb{Q})$. Write $\Tr(T^2)$ for the sum of the elements on the  diagonal
of $T^2$. Define $\Tr(Q^2)$ in the same way. It is easily checked that $\Tr(T^2)=\Tr(Q^2)$.
By the construction of the embedding $B_{p,\infty}\hookrightarrow M_4(\mathbb{Q})$, we have
\begin{equation}
\label{eq:TrU}
\Tr(U)=4\Tr(T^2).
\end{equation}

\begin{lemma}
\label{lem: Tr(alpha)}
Let $T$ denote the image of $\sqrt{\alpha}$ in $M_3(B_{p,\infty})$. Then $\Tr(T^2)=\Tr_{K^+/\Q}(\alpha)$.
\end{lemma}

\begin{proof}
Let $\alpha=\alpha_1,\alpha_2,\alpha_3$ denote the conjugates of $\alpha$. The characteristic polynomial of $U$ is $(X-\alpha_1)^{m_1}(X-\alpha_2)^{m_2}(X-\alpha_3)^{m_3}$ for some $m_1, m_2, m_3\in\mathbb{Z}_{> 0}$ with $m_1+m_2+m_3=12$. The trace of $U$ is $m_1\alpha_1+m_2\alpha_2+m_3\alpha_3\in\mathbb{Q}$.
If we can show that $m_1=m_2=m_3=4$, then equation \eqref{eq:TrU} gives
\begin{equation}
4\Tr(T^2)=\Tr(U)=m_1\alpha_1+m_2\alpha_2+m_3\alpha_3=4(\alpha_1+\alpha_2+\alpha_3)=4\Tr_{K^+/\Q}(\alpha).
\end{equation}
Therefore, it is enough to show that $m_1=m_2=m_3$. 
Since $\alpha\in\cO_{K^+}$, we have $\alpha_1+\alpha_2+\alpha_3\in\mathbb{Z}$ and therefore $(m_2-m_1)\alpha_2+(m_3-m_1)\alpha_3\in\mathbb{Q}$. Suppose for contradiction that we are not in the case $m_1=m_2=m_3$. Then, without loss of generality, $(m_2-m_1)\neq 0$ and since $\alpha_2\notin\Q$ it follows that $(m_3-m_1)\neq 0$. Therefore, $\alpha_3=\lambda\alpha_2$ for some $\lambda\in\mathbb{Q}$. But $\alpha_3$ is a Galois conjugate of $\alpha_2$ and the Galois group of the Galois closure of $K^+/\mathbb{Q}$ is either $C_3$ or $S_3$. Therefore, the automorphism sending $\alpha_2$ to $\alpha_3$ has order dividing $6$ and hence $\lambda$ is a sixth root of unity in $\mathbb{Q}$. Therefore, $\lambda=-1$ and $\alpha_3=-\alpha_2$. But this gives $\Tr_{K^+/\Q}(\alpha)=\alpha_1+\alpha_2+\alpha_3=\alpha_1$.  So $\alpha =\alpha_1 =  \Tr_{K^+/\Q} (\alpha)\in\mathbb{Q}$, which is a contradiction.
\end{proof}

\begin{proof}[Proof of Theorem \ref{thm:no embedding}]
Suppose for contradiction that $p>4\Tr_{K^+/\Q}(\alpha)^6/3^6$. We will show that the entries of the matrix $T$ commute, contradicting Lemma \ref{lem: noncommutative}. The key ingredients will be Lemma \ref{smallnorms} (which states that elements of a maximal order whose reduced norms are smaller than $\sqrt{p}/2$ commute) and equation \eqref{eq: sum of norms} below.

Recall that \begin{equation}
\label{eq: matrix T repeat}
T=\begin{pmatrix}\
 r & \delta_2 & \delta_3  \\
 -1 & svs^{\vee}/\delta_2 & swt^{\vee}/\delta_2  \\
 -1 & -tw^{\vee}s^{\vee}/\delta_3 & tzt^{\vee}/\delta_3
\end{pmatrix}
\end{equation}
where $\delta_2=\deg(s)$ and $\delta_3=\deg(t)$. We have
$$\begin{pmatrix}
1& 0 & 0 \\
0& \delta_2 & 0 \\
0& 0 & \delta_3
\end{pmatrix} T\in M_3(\End E_1).$$
We have chosen an isomorphism $\End^0 E_1\rightarrow B_{p,\infty}$, sending $\End E_1$ to a maximal order of $B_{p,\infty}$. The dual on $\End E_1$ corresponds to the usual involution on $B_{p,\infty}$. We identify $\End^0 E_1$ with $B_{p,\infty}$ and write $\Nrd(f)=\deg(f)=ff^{\vee}$ for $f\in\End E_1$.

By Lemma \ref{lem: Tr(alpha)}, we have $\Tr(T^2)=\Tr_{K^+/\Q}(\alpha)$. Writing out the entries on the  diagonal of $T^2$ gives
\begin{equation}
\label{eq: sum of norms}
0< \deg(r)+2\deg(s)+2\deg(t)+\deg(v)+2\deg(w)+\deg(z)=-\Tr_{K^+/\Q}(\alpha)<3\sqrt[6]{p/4}.
\end{equation}
Note that the sum of degrees is a sum of non-negative integers.
We want to use \eqref{eq: sum of norms} to bound the reduced norms of the non-scalar entries of 
$\begin{pmatrix}
1& 0 & 0 \\
0& \delta_2 & 0 \\
0& 0 & \delta_3
\end{pmatrix} T.$
 Recall that, in light of Lemma \ref{lem: st nonzero}, we are assuming that $s$ and $t$ are nonzero. Therefore, $\deg(s),\deg(t)\geq 1$ and \eqref{eq: sum of norms} gives
\begin{enumerate}
\item[i)] $\Nrd(r)=\deg(r)<3\sqrt[6]{p/4}-4<\sqrt{p}/2$,
\item[ii)]$2\deg(s)+\deg(v)< 3\sqrt[6]{p/4}$,
\item[iii)]  $2(\deg(s)+\deg(t)+\deg(w))< 3\sqrt[6]{p/4}$,
\item[iv)]$2\deg(t)+\deg(z)< 3\sqrt[6]{p/4}$.
\end{enumerate}
Observe that $\Nrd(swt^{\vee})=\deg(s)\deg(w)\deg(t)=\Nrd(-tw^{\vee}s^{\vee})$. So it remains to bound the reduced norms of $svs^{\vee}$, $swt^{\vee}$ and $tzt^{\vee}$.
Let $a\in\R_{>0}$. The maximum of the function $f(x)=x^2(a-2x)$ for $x\geq 0$ is achieved at $x=a/3$ and we have $f(a/3)=(a/3)^3$.
Applying this to ii) with $a=3\sqrt[6]{p/4}$, we see that
$$\Nrd(svs^{\vee})=\deg(s)^2\deg(v)<(\sqrt[6]{p/4})^3=\sqrt{p}/2.$$
Similarly, using iv) we get $$\Nrd(tzt^{\vee})=\deg(t)^2\deg(z)<(\sqrt[6]{p/4})^3=\sqrt{p}/2.$$
 Using iii), we get
\begin{eqnarray*}
\Nrd(swt^{\vee})=\deg(s)\deg(w)\deg(t)\leq(\deg(s)+\deg(w))^22\deg(t)
<(\sqrt[6]{p/4})^3=\sqrt{p}/2.
\end{eqnarray*}
Therefore, by Lemma \ref{smallnorms}, the entries of $\begin{pmatrix}
1& 0 & 0 \\
0& \delta_2 & 0 \\
0& 0 & \delta_3
\end{pmatrix} T$ commute. Since the entries of $\begin{pmatrix}
1& 0 & 0 \\
0& \delta_2 & 0 \\
0& 0 & \delta_3
\end{pmatrix} T$ are just scalar multiples of the entries of $T$, this means that the entries of $T$ commute. But this contradicts Lemma \ref{lem: noncommutative}. Therefore, the assumption $p>4\Tr_{K^+/\Q}(\alpha)^6/3^6$ does not hold.
\end{proof}

\subsection{Solutions to the embedding problem in the case that $K$ contains an
 imaginary quadratic subfield} 
\label{subsec:degenerate}

In this section, we consider the case where the sextic CM-field $K$
contains an imaginary quadratic subfield $K_1$. We show that the
embedding problem \ref{problem:embedding} has solutions for
every prime $p$ (Corollary \ref{cor:embedding}). The solutions are
constructed via the reduction at $p$ of a CM-abelian variety $A=E^3$ in
characteristic zero, where $E$ is an elliptic curve. In particular,
the CM-type of $A$ is imprimitive (Theorem \ref{thm: primsimple}). The
solutions we construct may therefore be called {\sl degenerate solutions} to the embedding problem.

The point is that if $K$ is a CM-field which contains an imaginary
quadratic subfield then there always exist imprimitive CM-types for
$K$. This is what allows for the existence of degenerate solutions to
the embedding problem.  Recall from Corollary \ref{cor: case3} that
there do not exist imprimitive CM-types $(K, \varphi)$ for CM-fields
that do not contain a proper CM-subfield.

The proof of Theorem \ref{thm:bnd_on_p} relied on showing
non-existence of solutions of the embedding problem for sufficiently
large primes (Theorem \ref{thm:no embedding}) in the case where the
sextic CM-field contains no proper CM-subfield. In contrast, if $C$ is
a curve whose Jacobian has CM by a sextic CM-field $K$ which contains
an imaginary quadratic field, then this strategy breaks down because
there the embedding problem has degenerate solutions for all primes
$p$ (Corollary \ref{cor:embedding}). 
The embedding problem, as formulated in Problem
\ref{problem:embedding}, does not take the CM-type into
consideration. It may be possible to prove an analogous result to
Theorem \ref{thm:bnd_on_p}, in the case that $K$ contains a proper
CM-subfield, using a more refined formulation of the embedding problem
that includes the CM-type as part of the data.

 \begin{proposition}
\label{prop:embedding existence}
Let $K$ be a sextic CM-field containing a proper CM subfield
$K_1$. Let $E$ be an elliptic curve over an arbitrary field and
suppose that there exists an embedding $K_1\hookrightarrow \End^0
(E)$.
Then there exists an order $\calO$ of $K$ and a ring embedding
\[\calO\hookrightarrow \End(E^3)= M_3(\End(E))\]
such that the Rosati involution on $\End(E^3)$ corresponding to the
product polarization on $A=E^3$ induces complex conjugation on $\cO$.
\end{proposition}

\begin{proof}
It suffices to give an injective $\Q$-algebra homomorphism
\begin{equation}
K\hookrightarrow \End^0(E^3)=M_3(\End^0(E)).
\end{equation}
This can be achieved as follows. Write $K=K^+K_1$ where $K^+/\mathbb{Q}$ is a totally real field with $[K^+:\mathbb{Q}]=3$. Choose a primitive element
$\alpha$ of $K^+/\Q$, so $K^+=\Q(\alpha)$. Embed $K_1$ diagonally via the fixed embedding of
$K_1$ into $\End^0(E)$. Map $\alpha$ to a symmetric matrix $Q\in M_3(\mathbb{Q})$ which has the same minimal polynomial as $\alpha$. Since all the conjugates of $\alpha$ are real, the existence of the matrix $Q$ is proved in Theorem 4 of \cite{Bender}. Extend to a $\mathbb{Q}$-algebra homomorphism. 
\end{proof}

In Remark~\ref{remark:CMconstruction}, we reviewed the construction in characteristic~$0$ of genus 3 curves with CM by a sextic CM-field $K$.  Similarly, when  $K_1$ is an imaginary quadratic field, elliptic curves with CM by $K_1$ exist in characteristic zero. 
For example, we may
take $E=\C/{\mathcal O}_{K_1}$, where we consider the maximal
  order ${\mathcal O}_{K_1}$ of $K_1$ as lattice in $\C$ (\cite{Silverman2}, Remark
    II.4.1.1). Then $\End(E)=\calO_{K_1}$. Moreover, $j(E)$ is an algebraic
    integer (\cite{Silverman2}, Theorem II.6.1). (This can be deduced
    from Theorem \ref{thm:SerreTate} which states that $E$ has
    potentially good reduction.) In particular, $E$ can be defined
    over the number field $M:=\Q(j(E))$.

We now show the existence of elliptic curves with CM by $K_1$ in positive characteristic. As above, $E/M$ is an
elliptic curve defined over the
number field $M$ with $\End(E)={\mathcal O}_{K_1}$.
 We choose a rational prime $p$, and let $\mathfrak{p}$ be a prime of
 $M$ above $p$. After extending $M$ if necessary, we may assume that
 $E$ has good reduction at $\mathfrak{p}$. Write
 $\overline{E}_\mathfrak{p}$ for the reduction of $E$ at
 $\mathfrak{p}$. We obtain an embedding
\[
{\mathcal O}_{K_1}=\End(E)\hookrightarrow \End(\overline{E}_\mathfrak{p}).
\]
This proves the  following lemma.

\begin{lemma}\label{lem:degenerate}
Let $p$ be a prime. Then there exists an elliptic curve
$\overline{E}_p$ in characteristic $p$ with
${\mathcal O}_{K_1}\hookrightarrow \End(\overline{E}_p)$.
\end{lemma}

The following result follows immediately from Lemma
\ref{lem:degenerate} and Proposition \ref{prop:embedding existence}.

\begin{cor}\label{cor:embedding}
Let $K$ be a sextic CM-field containing an imaginary quadratic field
$K_1$. Then there exists an order $\calO$ of $K$ for which there exists a solution to the embedding problem
for ${\mathcal O}$ and $p$ for every prime number $p$.
\end{cor}

Corollary \ref{cor:embedding} does not specify whether the elliptic
curve $\overline{E}_p$ from Lemma \ref{lem:degenerate} is ordinary
or supersingular.  The following proposition answers this
question. Note that it follows that the set of primes where the
elliptic curve $\overline{E}_p$ is supersingular has Dirichlet density
$1/2$.

\begin{proposition}(Deuring's Theorem)\label{prop:Deuring}
Let $E/M$ be an elliptic curve with CM by ${\mathcal O}_{K_1}$. Let
$p$ be a rational prime and $\mathfrak{p}$ be a prime of $M$ above $p$
such that $E$ has good reduction at $\mathfrak{p}$. Then the reduction
$\overline{E}_\mathfrak{p}$ of $E$ at $\mathfrak{p}$ is supersingular
if and only if $p$ is inert or ramified in $K_1$.  
\end{proposition}

Proposition \ref{prop:Deuring} is well known, but hard to find explicitly
in the literature. The statement can be proved using Theorem 10 of
Section 10.4 of \cite{LangEF}. We give the idea of the proof of the
proposition.  Let $\overline{E}/\F_q$ be an elliptic curve. Write
$\pi$ for its $q$-Frobenius endomorphism. Then $\overline{E}$ is
supersingular if and only if there exists integers $n,m$ such that
$\pi^n=[p]^m$, where $[p]$ denotes multiplication by $p$. (See for
example the proof of the Theorem of Deuring in Section 22 of
\cite{MAV}). The theorem from \cite{LangEF} shows that this happens if
and only if $p$ is inert or ramified in $K_1$.

\appendix

\section{Equations} 
\label{sec:equations}

 In this section, we list the equations obtained from a possible solution to the embedding problem. We start by setting some notation.

Let $K^{+}$ be the maximal real subfield of the sextic CM-field $K=K^{+}(\eta)$. 
Take an integral basis of $\calO_{K^+}$, so $\calO_{K^+}=\alpha_1\Z \oplus \alpha_2 \Z \oplus \alpha_3 \Z$.
We may assume that $K^+=\mathbb{Q}(\alpha_1)$.
We fix the
following notation:
\begin{itemize}
\item $\Tr_{K/K^+}(\eta)=a_1\alpha_1+a_2\alpha_2+a_3\alpha_3$
\item $\N_{K/K^+}(\eta)=b_1\alpha_1+b_2\alpha_2+b_3\alpha_3$
\item $f_i(x)=x^3+m_ix^2+n_ix+s_i$ is the characteristic polynomial of $\alpha_i$ over $\mathbb{Q}$ for $i=1,2,3.$
\end{itemize}

A solution to the embedding problem (Problem \ref{problem:embedding}) gives us three elliptic curves $E_1,E_2,E_3$ and an embedding of $\iota: \calO_K \hookrightarrow \End(E_1\times E_2\times E_3)$ such that Rosati involution on $E_1 \times E_2 \times E_3$ restricts to complex conjugation in the image of $\calO_K$. This gives the following conditions on $\iota(\alpha_i)$ and $\iota(\eta)$:

\begin{enumerate}
\item Commutativity: \begin{enumerate}
\item $\iota(\alpha_i) \iota(\eta)=\iota(\eta) \iota(\alpha_i)$ for all $i=1,2,3$.
\item $\iota(\alpha_i)\iota(\alpha_j)=\iota(\alpha_j)\iota(\alpha_i)$ for all $i \neq j \in \{1,2,3 \}$.
\end{enumerate}
\item Characteristic polynomial: $f_i(\iota(\alpha_i))=0$ for all $i=1,2,3$.
\item Norm: $\iota(\eta)\iota(\eta)^\dagger=b_1\iota(\alpha_1)+b_2\iota(\alpha_2)+b_3\iota(\alpha_3)$, where $\dagger$ denotes the conjugate transpose.
\item Trace: $\iota(\eta)+\iota(\eta)^\dagger=a_1\iota(\alpha_1)+a_2\iota(\alpha_2)+a_3\iota(\alpha_3)$.
\item Duality/Complex conjugation: $\iota(\alpha_i)=\iota(\alpha_i)^{\dagger}$ for all $i=1,2,3$. Since we are interested in the case that Rosati involution induces complex multiplication and since $\eta$ can be chosen so that $\eta^2 \in K^+$ is totally negative, we have $\iota(\eta)^{\dagger}=-\iota(\eta)$.
\end{enumerate}

  In the rest of this appendix, we will only write the conditions for
   $i=1$ which is enough if we have a power basis. In any case, the
   other relations for $i=2,3$ are similar. We now write the conditions above in terms of
matrix coefficients.  We are using the conventions and maps introduced in
Section \ref{sec:endo}.

Let $M=\iota(\alpha_1)$ be the matrix  $\begin{pmatrix}\
 a & b & c  \\
 d & e & f  \\
 g & h & \ell
\end{pmatrix}$ and $N=\iota(\eta)$ be the matrix $\begin{pmatrix}\
 p & q & r  \\
 s & t & u  \\
 v & w & y
\end{pmatrix} .$

 \subsubsection{Equations for duality/complex conjugation condition}\label{sec:duality}
  The relation $\iota(\eta)^{\dagger}=-\iota(\eta)$ translates into $M=M^{\vee}$ i.e.,  $\begin{pmatrix}\
   a & b & c  \\
   d & e & f  \\
   g & h & \ell
  \end{pmatrix} =\begin{pmatrix}\
   a ^{\vee}& d^{\vee} & g^{\vee}  \\
   b^{\vee} & e^{\vee} & h^{\vee}  \\
   c^{\vee} & f^{\vee} & \ell^{\vee}
  \end{pmatrix} . $ 

This gives us the following relations.

  \begin{remark} Note that we name the relations with respect to the variables we intend to use later on. Our aim is to simplify the equations and write everything in terms of the upper triangular entries of our matrices which are $a,b,c,e,f,\ell$ in the case of $M$ and $p,q,r,t,u,y$ in the case of $N$.
  \end{remark}

  \begin{enumerate}
  \item[(b-d)] $d=b^{\vee}$
  \item[(c-g)] $g=c^{\vee}$
  \item[(f-h)] $h=f^{\vee}$
  \item[(int)] $a,e,\ell$ are integral and in $\Q$, hence they are integers.

  \end{enumerate}

  The relation  $\iota(\eta)^{\vee}=-\iota(\eta)$  translates into: $\begin{pmatrix}\
   p & q & r \\
   s & t & u \\
   v & w & y
  \end{pmatrix} =\begin{pmatrix}\
     -p^{\vee} & -s^{\vee} & -v^{\vee} \\
     -q^{\vee} & -t^{\vee} & -w^{\vee}  \\
     -r^{\vee} & -u^{\vee} & -y^{\vee}
    \end{pmatrix} .$

  This gives us the following relations:

    \begin{enumerate}
    \item[(q-s)] $s=-q^{\vee}$
    \item[(r-v)]$v=-r^{\vee}$
    \item[(u-w)]$w=-u^{\vee}$
    \item[(trace)]$p=-p^{\vee}$, $t=-t^{\vee}$,  and $y=-y^{\vee}$\\
     i.e., $p$, $t$, and $y$ have trace zero in $\End(E_1)$, $\End(E_2)$, and $\End(E_3)$ respectively.
    \end{enumerate}

\subsubsection{Equations for commutativity condition} \label{sec:comm}
 Using $M$ and $N$ as above, the condition means $MN=NM$ which translates into the following equations:
 \begin{enumerate}
 \item[(i-i)] $ap + bs + cv = pa + qd + rg$.   (By equation (int) in Section \ref{sec:duality}, $a$ is an integer.  Hence $ap=pa$ and $bs + cv =  qd + rg$.)
 \item[(i-ii)] $aq + bt + cw= pb + qe + rh$
 \item[(i-iii)] $ar + bu + cy= pc + qf + r \ell $
 \item[(ii-i)] $dp + es + fv= sa + td + ug$
 \item[(ii-ii)] $dq + et + fw = sb + te + uh$ (By equation (int) in Section \ref{sec:duality}, $e$ is an integer.  Hence $et=te$ and $dq + fw = sb +  uh$.)
 \item[(ii-iii)] $dr + eu + fy= sc + tf + u\ell $
 \item[(iii-i)] $gp + hs + \ell v= va + wd + yg$
 \item[(iii-ii)] $gq + ht + \ell w = vb + we + yh$
 \item[(iii-iii)] $gr + hu + \ell y= vc + wf + y\ell $ (By equation (int) in Section \ref{sec:duality}, $\ell$ is an integer.  Hence $\ell y=y \ell$ and  $gr + hu = vc + wf$.)
 \end{enumerate}

 \subsubsection{Combining duality and commutativity conditions}

 Now we will plug in the equations we obtained in Section \ref{sec:duality} into the equations we obtained in Section \ref{sec:comm}. Note that our aim is to simplify the equations and write everything in terms of the upper triangular entries of our matrices which are $a,b,c,e,f,\ell$ in the case of $M$ and $p,q,r,t,u,y$ in the case of $N$.

 \vspace{1cm} \begin{tabular}{|c||c|} \hline Relation & Obtained
 using: \\ \hline\hline $bq^{\vee}+cr^{\vee}+rc^{\vee}+qb^{\vee}=0 $ &
 (i-i), (c-g), (q-s), (v-r) \\ \hline
 $pb+qe+rf^{\vee}-aq-bt+cu^{\vee}=0$ & (i-ii), (u-w), (f-h) \\ \hline
 $ar + bu + cy- pc - qf - r \ell=0 $& (i-iii) \\ \hline $b^{\vee}p -
 eq^{\vee} - fr^{\vee}+q^{\vee}a - tb^{\vee} - uc^{\vee}=0 $& (ii-i),
 (b-d), (q-s), (r-v), (q-s) \\ \hline $b^{\vee}q - fu^{\vee}
 +q^{\vee}b - uf^{\vee}=0$ & (ii-ii), (b-d), (u-w), (q-s),
 (f-h) \\ \hline $dr + eu + fy+q^{\vee}c - tf - u\ell=0 $ & (ii-iii),
 (q-s) \\ \hline $c^{\vee}p - f^{\vee}q^{\vee} + (a-\ell)r^{\vee} +
 u^\vee b - yc^\vee=0$ & (iii-i), (c-g), (f-h), (s-q), (r-v), (u-w),
 (b-d), (int) \\ \hline $c^{\vee}q + f^{\vee}t + (e- \ell)u^{\vee}
 +r^{\vee}b - yf^{\vee}=0$ & (iii-i), (c-g), (f-h), (u-w), (r-v),
 (int)\\ \hline $c^{\vee}r + f^{\vee}u + r^{\vee}c + u^{\vee}f=0$ &
 (iii-i), (f-h), (u-w), (r-v) \\ \hline \end{tabular}

 \subsubsection{Equations for characteristic polynomial
condition}\label{sec:char} The characteristic polynomial condition for
$i=1$ translates into $0=M^3+m_1M^2+n_1M+s_1$.  Combining this
equality with Equation (int) of Section \ref{sec:comm} gives the
following equations. For instance, for the top left corner of the
matrix sum we get
\[
0=a^3+abd+acg+bda+bed+bfg+cga+chd+c\ell g +
m_1(a^2+bd+cg)+n_1a+s_1 .
\]
 If we apply Condition (int) this turns into
\[
 (2a+e+m_1) bd+(2a+\ell + m_1)cg+bfg+chd+a^3+m_1a^2+n_1a+s_1=0.
 \]
  The
following is the list of equations coming from all nine entries.

\begin{enumerate}[label=(\roman*)]

\item $ (2a+e+m_1) bd+(2a+\ell + m_1)cg+bfg+chd+a^3+m_1a^2+n_1a+s_1=0$
 \item $ (a^2+ae+e^2+m_1a+m_1e+n_1)b+(e+\ell+m_1+a)ch+bdb+bfh+cgb=0$
  \item $ (a^2+a\ell+\ell^2+m_1a+m_1\ell+n_1)c+(a+e+\ell+m_1)bf+bdc+cgc+chf=0$
\item $(a^2+ea+e^2+m_1a+m_1e+n_1)d+(e+a+\ell + m_1)fg+dbd+dcg+fhd=0 $
\item $ (a+2e+m_1) db+(2e+\ell+m_1)fh+dch+fgb+e^3+m_1e^2+n_1+s_1=0$
\item$ (a+\ell+e+m_1)dc+(e^2+e\ell+\ell^2+m_1e+m_1\ell+n_1)f+dbf+fgc+fhf=0$
\item$ (a^2+\ell a+\ell^2+m_1a +m_1\ell +n_1)g+(a+e+\ell+m_1)hd+gbd+gcg+hfg=0$

\item$ (e^2+\ell e+\ell^2+m_1e+m_1\ell+n_1)h+(a+e+\ell+m_1)gb+gch+hdb+hfh=0$

\item$ (a+2\ell+m_1)gc+(e+2\ell+m_1)hf+gbf+hdc+\ell^3+m_1\ell^2+n_1\ell+s_1=0$
\end{enumerate}

  \subsubsection{Combining duality and characteristic polynomial conditions}

   Now we will plug in the equations we obtained in Section \ref{sec:duality} into the equations we obtained in Section \ref{sec:char}. Note that our aim is to simplify the equations and write everything in terms of the upper triangular entries of our matrices which are $a,b,c,e,f,\ell$ in the case of $M$ and $p,q,r,t,u,y$ in the case of $N$. Note that $\Nrd(x)=xx^{\vee}, \Tr(x)=x+x^{\vee}$ denote the reduced norm and trace of an element. Since the norm and trace are scalars, they commute with everything else.

   We start with the relations coming from $M$:

 \begin{enumerate}[label=(\Roman{*}), ref=(\Roman{*})]
 \item \label{cpm11} 
 $(2a+e+m_1) \Nrd(b) + (2a+\ell + m_1) \Nrd(c) + \Tr(bfc^{\vee})+a^3+m_1a^2+n_1a+s_1=0$

 \item \label{cpm12} 
  $ (a^2+ae+e^2+m_1a+m_1e+n_1+\Nrd(b)+\Nrd(c)+\Nrd(f))b+(a+e+\ell+m_1)cf^{\vee}=0 $

 \item \label{cpm13}
 $(a^2+a\ell+\ell^2+m_1a+m_1\ell+n_1+\Nrd(b)+\Nrd(c)+\Nrd(f))c+(a+e+\ell+m_1)bf=0 $

 \item \label{cpm21} 
 $(a^2+ae+e^2+m_1a+m_1e+n_1+\Nrd(b)+\Nrd(c)+\Nrd(f))b^{\vee}+(a+e+\ell + m_1)fc^{\vee}=0$

 \item \label{cpm22}  
 $(a+2e+m_1) \Nrd(b)+(2e+\ell+m_1)\Nrd(f)+\Tr(b^{\vee}cf^{\vee})+e^3+m_1e^2+n_1e+s_1=0$

 \item \label{cpm23} 
 $(e^2+e\ell+\ell^2+m_1e+m_1\ell+n_1+\Nrd(b)+\Nrd(c)+\Nrd(f))f+ (a+\ell+e+m_1)b^{\vee}c=0
 $

 \item \label{cpm31} 
 $ (a^2+a\ell +\ell^2+m_1a +m_1\ell +n_1+\Nrd(b)+\Nrd(c)+\Nrd(f))c^{\vee}+(a+e+\ell+m_1)f^{\vee}b^{\vee}=0
 $

 \item \label{cpm32} 
 $(e^2+e\ell+\ell^2+m_1e+m_1\ell+n_1+\Nrd(b)+\Nrd(c)+\Nrd(f))f^{\vee}+(a+e+\ell+m_1)c^{\vee}b=0$

 \item \label{cpm33}  
 $(a+2\ell+m_1)\Nrd(c)+(e+2\ell+m_1)\Nrd(f)+\Tr(c^{\vee}bf)+\ell^3+m_1\ell^2+n_1\ell+s_1=0$

  \end{enumerate}

Write $\Tr(X)$ for the sum of the entries on the main diagonal of a matrix $X$. Notice that if we take $\eta=\sqrt{\alpha_1}$ like in Section \ref{62}, then
$$
-m_1=\Tr(\alpha_1)=\Tr(N^2)=\Tr(M)=a+e+\ell,
$$ where the first equality follows by definition, the second equality is
Lemma \ref{lem: Tr(alpha)}, the third equality holds because we took $\eta=\sqrt{\alpha_1}$, and the final equality is the definition of $\Tr(M)$.
This implies that
Equation \ref{cpm12} = Equation \ref{cpm21}, Equation \ref{cpm13}=
Equation \ref{cpm31} and
Equation \ref{cpm23}=Equation \ref{cpm32}.

Combining $-m_1=a+e+\ell$ with relations \ref{cpm11}-\ref{cpm33}, we deduce the following relations on the coefficients $m_1,n_1,s_1$ of the characteristic polynomial of $\alpha_1.$
\begin{enumerate}
\item\label{m1}
$m_1=-(a+e+\ell)$
\item
$n_1=ae+e\ell+a\ell-\Nrd(b)-\Nrd(c)-\Nrd(f)$ (using Equation (\ref{m1}) together with Equations \ref{cpm12}, \ref{cpm13} and \ref{cpm23}.)
\item
$s_1=a\Nrd(f)+e\Nrd(c)+l\Nrd(b)-ae\ell-\Tr(bfc^{\vee})$ (using Equation (\ref{m1}) together with Equations \ref{cpm11}, \ref{cpm22} and \ref{cpm33}.)
\end{enumerate}

\bibliographystyle{plain}

\end{document}